\documentclass[12pt]{article}
\usepackage[a4paper,margin=2.5cm]{geometry}
\usepackage[active]{srcltx}
\usepackage{geometry}

\usepackage{amsfonts}
\usepackage{amssymb}
\usepackage{amsfonts}
\usepackage{amsthm}
\usepackage{bbm}
\usepackage{bm}

\usepackage{longtable}

\usepackage{times}
\usepackage{color,cite,graphicx}
\fussy
\usepackage{chapterbib}

\def\[{\begin{equation}}
\def\]{\end{equation}}
\newcounter{numer1}

\usepackage{multirow}
\usepackage{setspace}

\usepackage{graphicx}
\usepackage{amssymb}
\usepackage[all]{xy}

\usepackage{epsfig}

\newtheorem{theorem}{Theorem}
\newtheorem{lemma}{Lemma}
\newtheorem{propo}{Proposition}

\newtheorem{definition}{Definition}
\newtheorem{assumption}{Assumption}
\newtheorem{OBS}{Remark}

\usepackage{color}

\title{ Dynamics of Coupled Maps in Heterogeneous Random Networks}
\author{Tiago Pereira\footnote{tiago.pereira@imperial.ac.uk}, ~  Sebastian van Strien\footnote{s.van-strien@imperial.ac.uk}, ~ and Jeroen S.W. Lamb\footnote{jsw.lamb@imperial.ac.uk} \\
Department of Mathematics, Imperial College London\\ }

\begin{document}
\maketitle

\begin{abstract}
We study expanding circle maps interacting in a heterogeneous random network. ~ ~ Heterogeneity means that some nodes in the network are massively
connected, while the remaining nodes are only poorly connected.  We provide a
probabilistic approach which enables us to describe the effective dynamics of the massively connected nodes when taking a weak interaction limit. More precisely, 
we show that for almost every random network and almost all initial conditions 
the high dimensional network governing the dynamics of the massively connected nodes   can be reduced to a few macroscopic equations. Such reduction
is intimately related to the ergodic properties of the expanding maps. This reduction allows one to explore the coherent properties of the network. 
\end{abstract}


\section{Introduction}

Understanding the behavior of interacting dynamical systems is a long standing problem. 
Main efforts focused on dynamical systems interacting in a lattice, 
and with a mean field coupling \cite{Kaneko,Keller}. Typical attempts establish the existence of an absolutely 
continuous invariant measure, see for example \cite{Sinai,Jiang, Rugh, Jar, BricmontKupiainen,KellerLiverani1,KellerLiverani2,Book}. More recently, the attention has shifted towards more general and 
irregular networks of coupled maps \cite{Young}. 

The last decade has witnessed a rapidly growing interest in dynamics on spaces 
with {\lq}complex topology{\rq}. This interest is partly motivated by 
results showing that the structure of a network  can 
dramatically influence the dynamical properties of the system, but also because
many, disparate, real-world networks share a common feature -- heterogeneity in the interaction structure \cite{Albert,Newman}. 
This suggests a network structure  in which most nodes have degree close to the minimum, while some  high-degree nodes, termed {\em hubs}, are present and have greater impact upon 
the network functioning.  Recent work has suggested the ability of networks with a heterogeneous degree distribution 
to present degree dependent collective behavior \cite{Zhou,Arenas,HubSync,Baptista}. 
Hubs may undergo a transition to coherence whereas the remaining nodes behave incoherently. 

The dynamical properties among the hubs play a central role in many realistic 
networks. There is evidence that the dynamics of hub neurons coordinate and shape the network 
in a developing hippocampal network \cite{HubScience}, and play a 
major role in epileptic seizures \cite{HubEp}.  
Recent efforts to understand the dynamics of the hubs 
concentrated mainly on numerical investigations \cite{Zhou,Arenas,HubSync,Baptista}. 
The description of the dynamical properties of the hubs remains elusive. 

It remains unknown how hub dynamics depends on various network parameters such as
the dynamics of the isolated nodes, and the structure of the network. Revealing
the dynamics of the hubs in relation to the graph structure is an important step toward understanding 
complex behavior and unveiling the topological implications on the network 
functioning. 

In this paper, we study the dynamics of coupled expanding circle maps on a 
heterogeneous random network. We provide a probabilistic approach to describe 
effective dynamics of the massively connected nodes in a weak interaction limit. We 
show that for almost every random network and almost all initial conditions 
the high dimensional network problem governing the dynamics of the massively connected nodes  
can be reduced to a few macroscopic equations. Such reduction
is intimately related to the ergodic properties of the expanding maps. This reduction allows one to 
explore the coherent properties of the network. Our analysis reveals that the intrinsic properties of the 
node dynamics also play a major role and may hinder or enhance 
coherence.

\section{Notation and Statement of the Results}
~ 

Our main object of study is the network dynamics of coupled maps.  
{\it A network of coupled dynamical systems} is defined to be a triple $(G,f_i,h)$ where:
\begin{itemize}
\item[] $G$ is a labelled graph of $n$ nodes, termed network; see Section \ref{Nets} for details. 
\item[] $f_i : M_i \rightarrow M_i$ is the local dynamics at  each node in the network $G$, see Section \ref{dyn}.
\item[] $h$\, described the coupling scheme, see Section \ref{Int}.
\end{itemize}
Abstractly, the network dynamics is defined by the iteration $\Phi : X \rightarrow X$, where $X:= M_1 \times \cdots \times M_n$
is the product space and $\Phi = A \circ F$, where $F = f_1 \times \cdots \times f_n$, and 
$A : X \rightarrow X$ defines the spacial interaction $G$ and type of coupling  interaction scheme $h$.  
In Section \ref{netdyn} we shall present the above network equation from the single node perspective. 
In what follows we define our local dynamics and the network class we are working with. Finally, we  define the class of 
coupling functions of interest. For simplicity we introduce the following 

{\bf Notation: } Given functions $a,b\colon \mathbb{R} \to \mathbb{R}$ (or sequences $a_n,b_n$),
we write $a\asymp b$ (resp. $a\lesssim b$) if there exists a universal constant  $C$ so that
$1/C\le |a/b|\le C$ (resp. $|a|\le C|b|$). Likewise, denote $a\gtrsim b$ if $|a|\ge C|b|$.
For simplicity, otherwise unless stated, we understand the sums running over $1$ to $n$. 
Throughout we denote by  $\|\zeta\|_0$  the $C^0$-norm of a function $\zeta$.

\subsection{Random Networks}\label{Nets}

We concentrate our attention on networks of $n$ nodes described by labelled graphs. 
Our terminology is that of Refs. \cite{Chung,Bollobas}.
We regard such graphs as networks of size $n$.  We use a random network model $\mathcal{G}(\bm{w})$ which is an extension of the Erd\"os-R\'enyi 
model for random graphs with a general degree distribution, see for example Ref. \cite{ChungComb}.
Here $\bm{w} = \bm{w}(n)$,
$$
{\bm w} = (w_1 , w_2 ,  \cdots, w_n ),
$$
will describe the expected degree of each node; 
for convenience we order $w_1 \ge w_2 \ge \cdots \ge w_n\ge 0$, and denote $w_1 = \Delta$. 
In this model $\mathcal{G}(\bm{w})$ consists of the space of all
graphs of size $n$, where each  
potential edge between $i$ and $j$ is chosen with probability 
$$
p_{ij} = w_i w_j \rho,
$$ 
and where 
$$\rho = \frac{1}{\sum_{i=1}^n w_i}.$$ To ensure that $p_{ij} \le 1$
it assumed that $\bm{w}=\bm{w}(n)$ is chosen so that 
\begin{equation}
\label{DeltaRho}
\Delta^2  \rho \le 1 .
\end{equation}
Note that the model $\mathcal{G}(\bm{w})=\mathcal{G}(\bm{w}(n))$ is actually a probability space, 
where the sample space is the finite set of networks of size $n$ endowed with the power set $\sigma$-algebra. Moreover, 
the probability measure $\mbox{Pr}$ on the sample space is generated by $p_{ij}$. 

Throughout the paper, we will take expectation with respect to measures associated with the node dynamics. 
Therefore, if for clarity we need to emphasize that the probability and expectation are
taken in $\mathcal{G}(\bm{w})$, we write for a given $C \in \mathcal{G}$ and for a random variable 
$X$,  
$$
\mbox{Pr}_{\bm{w}}(C)=\mbox{Pr}_{\bm{w}(n)}(C) \mbox{   and   }   \mathbb{E}_{\bm{w}}(X) = \mathbb{E}_{\bm{w}(n)}(X).
$$

{\bf Network Property:} We call a subset $Q \subset \mathcal{G}(\bm{w})$ a property of networks of order $n$ if transitivity holds:
if $G$ belongs to $Q$ and $H$ is isomorphic to $G$ (this means that the graphs are 
the same up to relabelling of the nodes) then $H$ belongs to $Q$ as well. 
We shall say that {\it almost every} network $G$ in $\mathcal{G}({\bm{w}(n)})$ has a certain property 
$Q$ if
\begin{equation}
\mbox{Pr}_{\bm{w}(n)} \left( G \mbox{  has property  } Q \right) \rightarrow 1
\label{AEnetwork}
\end{equation}
as $n \rightarrow \infty$. 
The assertion {\it almost every} $G \in \mathcal{G}(\bm{w})$ has property $Q$ is the same as 
the proportion of all labelled graphs of order $n$ that satisfy $Q$ tends to 1 as $n \rightarrow \infty$.

These networks may be described  in terms of its {\it adjacency 
matrix }  $A$, defined as 
$$
A_{ij} =
\left\{
\begin{array}{cc}
1 & \mbox{  if nodes $i$ and $j$ are connected  } \\
0 & \mbox{  otherwise  } \\
\end{array}
\right.
$$
In the model $\mathcal{G}$ each element of the adjacency $A_{ij}$'s 
is an independent Bernoulli (random) variable, taking  value $1$ with success probability $p_{ij}$.  
The {\em degree} $k_i$ of the $i$th node is the  number of connections it receives. $k_i$ is a 
random variable, which in terms of the adjacency matrix 
reads $k_i = \sum_j A_{ij}$. Note that as the network is oriented the number of connections $k_i^{out}$
the $i$th node makes with its neighbors need not to be $k_i$. Notice that $k_i^{out} = \sum_{j}A_{ji}$.
An interesting property of this model is that under this construction $w_i$
is the expected value of $k_i$, that is, $\mathbb{E}_{\bm{w}}(k_i) = w_i$ while 
also $\mathbb{E}_{\bm{w}}(k_i^{out}) = w_i$.

We are interested in the large size behavior of heterogeneous networks. Here we say that  a network is 
{\em heterogeneous} when there is a considerable disparity between the node's degree. Real 
world networks are typically heterogeneous -- a small fraction of nodes is
massively connected whereas the remaining nodes are only poorly connected.  To be precise, 
we study the following class

\begin{definition}[Strong Heterogeneity]
\label{H1} 
We say that the 
model $\mathcal{G}(\bm{w}(n))$ is strongly heterogeneous if the following 
hypotheses are satisfied:  
\begin{description}
 \item{N1 --} {\it Massively connected Hubs:} There is $\ell \in \mathbb{N}$ such that 
for every $1 \le i \le \ell$ if we write
$$
{w_{i,n}} = \kappa_{i,n} {\Delta}_n
$$
then  $\kappa_{i,n} \rightarrow \kappa_{i,\infty}$ as $n \rightarrow \infty$ and $\kappa_{i,\infty} \in  (0,1]$ (and $\kappa_{1,n}=1$).
We regard $\kappa_{i,n}$ as the normalized degrees.
\item{N2 --} {\it Slow growing low degrees :} for  $\ell < i \le n$ 
$$
\log n  ~\lesssim ~ w_{i,n}  ~ \lesssim ~ \Delta_n^{1- \gamma} 
$$
where $0<\gamma < 1$ control the scale separation between low degree nodes and the hub nodes. 
\item{N3 --} {\it Cardinality of the Hubs:} The number of hubs $\ell  = \ell(n)$ satisfies 
$$
{\ell}  \lesssim  \Delta^{\theta}_n
$$ 
for some small $\theta \le 1$.
\end{description}
\label{ws}
\end{definition}

We will often suppress the $n$ dependence in the notation so write
$\bm{w}$ instead of $\bm{w}(n)$ and we also often write $\Delta$ and $w_i,\kappa_i$ instead
of $\Delta_n$ and $w_{i,n},\kappa_{i,n}$.

The hypothesis that the number of highly connected nodes is significantly smaller than the system 
size,  $\theta < 1$, implies that the network is heterogeneous. The $\ell$ high-degree nodes are 
termed {\em hubs}, and the remaining nodes are called {\em low degree nodes}. 
These assumptions on heterogeneity mimic conditions observed in a large class of realistic networks
including neuronal networks, social interaction, internet,  among others \cite{Albert,Newman,HubScience}. Typically, the number of hubs in the network is much smaller than the system size (and the degrees of the hubs).
We wish to treat 
the large networks, so we assume that $n$ is large. That is, the networks we consider have finite but 
large number of nodes. 
\begin{figure}[htbp] 
   \centering
   \includegraphics[width=4in]{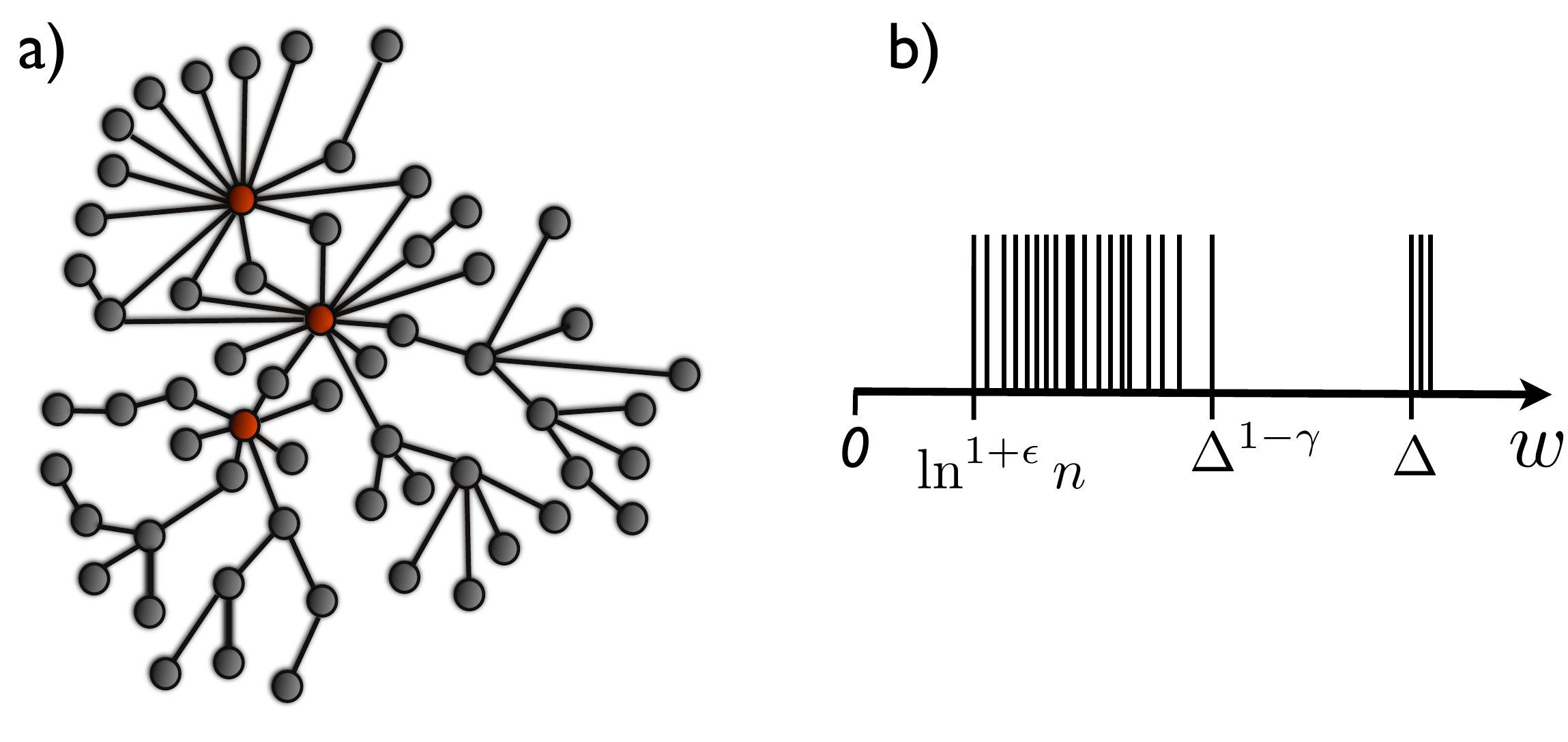} 
   \caption{The strong Heterogeneity condition. In a) a heterogeneous network, the hubs are depicted in red. While 
   the hubs are massively connected most nodes have only a few connections. In b) a pictorial presentation 
   of {\it N1} and {\it N2}. The axis denotes the expected degrees $w_i's$ ordered according to their magnitudes. 
   Whilst the hubs increase proportionally to $\Delta = \Delta_n$ the remaining nodes increase in another scale.%
   }
  \label{LambdaAlpha}
\end{figure}

An interesting property of such random graphs is that under the condition N2 the degrees have good 
concentration properties, as the next result shows
 
\begin{propo}[Concentration]
Let $\mathcal{G}(\bm{w})$ satisfy the strong heterogeneity condition. Then almost every random network $G \in \mathcal{G}(\bm{w})$   has every vertex satisfying 
\begin{equation}
| k_i - w_i |< w_i^{1/2 - \varepsilon} \nonumber
\label{ConIn}
\end{equation}
for any $\varepsilon > 0$.
(This statement should be understood in the sense of  (\ref{AEnetwork}).)
\end{propo}
\begin{proof}
We obtain this result via a Chebyshev inequality.
The expected value of $k_i = \sum_j A_{ij}$ is given by 
$\mathbb{E}_{\bm{w}}(k_i) = \sum_{j} w_i w_j \rho = w_i$. 
To estimate the variance of Var$_{\bm{w}} (k_i)$ notice that 
\begin{equation}\label{E2k}
\mathbb{E}_{\bm{w}}^2(k_i) = w_i^2 = w_i^2 \rho^2\sum_{j,k}  w_j w_k 
\end{equation}
and that 
$ \mathbb{E}_{\bm{w}}(k_i^2) =  \mathbb{E}_{\bm{w}}(\sum_{j,k} A_{ij}A_{ik})$. 
If $j\not=k$ then $A_{ij}$ and $A_{ik}$ are independent,  otherwise notice that $A^2_{ij} = A_{ij}$. This 
remark leads to 
\begin{equation}\label{Ek2}
\mathbb{E}_{\bm{w}}(k_i^2) = w_i + w_i^2 \rho^2 \sum_{j\not=k} w_i w_k  
\end{equation}
Combining (\ref{Ek2}) and (\ref{E2k}) we obtain 
$$
\mbox{Var}_{\bm{w}}(k_i) = w_i - w_i^2 \rho^2 \sum_j w_j^2 .
$$
This implies $\mbox{Var}_{\bm{w}}(k_i) <  w_i $. 
(In fact, the previous equation gives
$(1-p_{i,1})w_i\le \mbox{Var}_{\bm{w}}(k_i) < w_i$  where $p_{i1}\le 1$ is the probability to connect node $i$ to the main hub node $1$
since $w_j^2 \le \Delta w_j$ gives $\sum_j w_j^2 \le \Delta/ \rho$ and therefore
$w_i^2 \rho^2 \sum_j w_j^2 \le w_i p_{i1}$.)
Hence the Chebyshev inequality yields
$$
\mbox{Pr}_{\bm{w}} (|k_i - w_i | \ge w_i^{1/2 - \varepsilon}) \le w_i^{~ -2 \varepsilon}, 
$$
and condition N2 implies that this probability tends to zero as $n$ increases.  
Therefore, we obtain that $|k_i - w_i | \le w_i^{1/2 - \varepsilon}$ for almost every network, which concludes the proof. 
\end{proof}

\begin{OBS}\label{nDelta} This model imposes minimum and 
maximum growth conditions for $\Delta_n$ in terms of the network size $n$.
Indeed (N2) implies
$$
\Delta_n \gtrsim (\log n)^{\frac{1}{1-\gamma}}
$$
Moreover,   (\ref{DeltaRho}) leads to an upper estimate 
for the growth of the maximum expected degree $\Delta_n$, namely 
$$
\Delta_n \lesssim \max\big\{n^{1/2 + \delta} , n^{\frac{1}{1+ \gamma}}\big\}
$$
for some $\delta>0$ small enough.
\end{OBS}

The next example provides an illustration of a degree sequence satisfying the strong heterogeneity condition. We shall 
construct networks from this illustration for numerical simulations later on in the paper.
\medskip

\noindent
{\bf Example 1:} An example of a degree sequence satisfying the strong heterogeneity condition is 
$$
\bm{w} = (\kappa_1 \Delta, \dots , \kappa_{\ell} \Delta, w_{\ell+1} , \dots, w_n)
$$
with $1=\kappa_1  \ge \kappa_2 \ge \dots \ge \kappa_{\ell}>0$, 
so that  
$$\Delta = n^{\sigma}\mbox{ , } \ell = \Delta^{\theta} \mbox{ and }
w_i \asymp \Delta^{1 - \gamma} \mbox{ for }i\ge \ell$$
where $\gamma , \sigma, \theta\in (0,1)$.
Condition  (\ref{DeltaRho}) imposes a growth condition on the expected degree 
with scaling coefficient  
$$
\sigma < \frac{1}{1+\gamma}.
$$

It easy to modify this example so that the expected degree sequence exhibits a power law behavior in the distribution 
of the expected degrees, and other non trivial distributions. 
 
\subsection{Local Dynamics} \label{dyn}

We choose the dynamics on each node of the network to be identical $f_i = f$. 
In fact, it turns out that under our hypothesis 
if the dynamics on each node is slightly different our claims still hold true. We consider expanding maps on 
the circle $M= \mathbb{R} / \mathbb{Z}$ as a model of the isolated dynamics of the nodes. 
We will use that $M$ is  compact and the addition structure coming from $\mathbb{R}$, see Ref. \cite{vStrien} for details.
\begin{assumption}
\label{H2}
Let $f : M \rightarrow M$ be a $C^{1 + \nu}$ H\"older continuous expanding map, for some $\nu \in (0,1]$.  That is, we assume that there exists $\sigma > 1$ such that 
$$
\| D f(x) v\| \ge \sigma \| v \| 
$$
for all $x\in M$  and $v \in T_x M$, for some riemannian metric $\| \cdot \|$, and $Df$ is H\"older continuous 
with exponent $\nu$.
\end{assumption}

The differentiability condition, that is, $\nu > 0$, plays an important role in our analysis. It is well known that 
if $\nu = 0$, then $f$ may admit invariant measures which are singular with respect to Lebesgue measure. 
Moreover, if $\nu>0$ the system is structurally stable. 

\subsection{Interaction Function} \label{Int}

Our aim in this subsection is to introduce the interaction structure we will use in the network dynamics (\ref{md1}). We consider pairwise interaction 
$$
h : M \times M \rightarrow \mathbb{R}.
$$
For simplicity we assume $h$ satisfies the following 
representation
\begin{equation}\label{h}
h(x,y) = \sum_{p,q=1}^k u_p(x) v_q(y),
\end{equation}
where $k$ is an integer, and $u_p,v_q\colon M\to \mathbb R$ are $C^{1+\nu}$ functions (i.e.\ for each $x,y\in \mathbb R$ and $r,s\in  \mathbb{Z}$
one has  $u_p(x + r) = u_p(x)$ and $v_q(y+s) = v_q(y)$). Notice that
$h$ is well-defined. Of particular interest in applications is the interactions akin to diffusion, 
$$
h(x,y) = \sin 2\pi (x-y) \, \, \mbox{  and   }   \, \,  h(x,y) = \sin 2\pi x - \sin 2\pi y.
$$

With this interaction function we are ready to introduce the network dynamics.

\subsection{Network Dynamics} \label{netdyn}

Given an integer $n$ and a $n\times n$ interaction matrix $A$, 
we consider the dynamics of a network of $n$ coupled maps is described by 
$$F\colon M^n\to M^n\mbox{ where }(x_1(t+1),\dots,x_n(t+1))=F(x_1(t),\dots,x_n(t))$$
is defined by
\begin{equation}
{x}_i(t+1) = {\hat f}({ x}_i(t)) + \frac{\alpha}{\Delta} \sum_{j=1}^n A_{ij} h(x_j(t), x_i(t))\,  (\mbox{mod} \,1) , \,\, \, \,\mbox{ for } i=1,\dots,n.
\label{md1}
\end{equation}
Here  $\hat f\colon M\to \mathbb R$ is the lift of $f\colon M\to M$ (note that $M= \mathbb{R} / \mathbb{Z}$).
The right hand side in (\ref{md1}) is well-defined because $M$ has an addition structure,  $h\colon M\times M\to \mathbb R$ is well-defined and
for each choice of lift $f\colon M\to M$ to a map $\hat f\colon M\to \mathbb R$ the  expression in (\ref{md1}) gives the same result.  By abuse of notation we will denote the lift $\hat f$ also by $f$.
Here $A_{ij}$ is chosen as in Section~\ref{Nets}, 
 ${x}_i(t)$ describes the state of the $i$-th node at time (which has degree 
$k_i=\sum_{j}A_{ij}$) and $\Delta=w_1$ is the largest expected degree.
Moreover,  $\alpha$ 
is a free parameter which describes  the normalized overall coupling strength. We shall also denote the state $x_i(t)$
as $x_i$ whenever convenient if there is no risk of confusion.   
We are interested in the weak coupling limit as $n\to \infty$, 
that is, when $\alpha$ is independent of $\Delta$ and $n$.

\section{Main Result and Discussions}

Our main goal  is to obtain a low dimensional equation to describe the highly connected nodes. 
Our strategy is to prove the following reduction by means of an effective dynamics. 
For simplicity we choose the initial conditions on $M^n$ to be independent and identically distributed. More precisely

{\it Choice of Initial conditions:} Let 
$\mu$ be a measure supported on $M$ with density $\varphi$. Moreover, let $\log \varphi$ be $(a,\nu)-$H\"older where $\nu\in (0,1]$ (this notion is defined in Section~\ref{IsoN}). Consider the global phase space of the coupled maps $M^n$, and a product measure  $\mu^n : M^n \rightarrow [0,1]$ be   given such that $A = A_1 \times \cdots \times A_n \subset M^n$ we have $\mu^n(A) = \mu(A_1) \cdots \mu(A_n)$. This choice of initial conditions is natural in numerical experiments. We shall use this choice to state our main result

\begin{theorem}[Dynamics of Hubs]\label{thm:main}
Let $\mathcal{G}(\bm{w})$ be strongly heterogeneous and consider the coupled map network  (\ref{md1}) on $\mathcal{G}(\bm{w})$. Assume that the network structural parameters satisfy 
$$
0< \theta < 1 - \gamma \nu /2,
$$
and 
let all initial conditions on $M^n$ be given according to a product measure $\mu^n$.    
Then there exists a positive
number $u = u (\varphi)$ such that for almost every network in $\mathcal{G}(\bm{w})$ and for $\mu^n-$almost every initial condition the dynamics of the hubs $i=1,\dots, \ell$ is reduced to 
\begin{equation}
{x}_i(t+1) = { f}({ x}_i(t)) + \alpha \kappa_i g(x_i) + \alpha{\zeta}_i(t)  \mbox{~ (mod1) } , \mbox{ for each } t\ge u
\end{equation}
where $\kappa_i$ is the normalized degree, and
$$
g(x) = \int h(x,y) \mu_0 (dy)
$$
where $\mu_0$ is the invariant measure of $f$. Moreover, ${\zeta}_i(t)$ satisfies
$$
\| {\zeta}_i \|_0 \lesssim \kappa_i^{1/2} \Delta^{-\gamma \nu / 2 + \delta} \mbox{ for each } t\ge u
$$
for any $\delta>0$ (where $\|\zeta_i\|_0$ is the $C^0$-norm of $\zeta_i$).  \,
\label{MFR}
\end{theorem}

\begin{OBS} If  $0<\theta < 1 - \gamma \nu /2$ is not satisfied the function $\zeta$ stills 
converges to zero as $\Delta \rightarrow \infty$. The speed of convergence will 
then depend on the parameters of the network, see Proposition \ref{emf}. 
\end{OBS} 
 
Our result has a mean field flavor. Loosely speaking, the hubs interact with their local 
mean field. As the network is random the local mean field perceived by the hubs is a fraction of the 
total mean field with weights given by the normalized degrees $\kappa_i$. 
On the other hand, 
the mean field is determined by the invariant measure $\mu_0$ of the isolated dynamics.  
The function ${\zeta}_i$ describes the noise-like disturbance on the motion of the hubs.

{\it $\mu^n-$almost every initial condition:} The claim concerning the reduction holds 
almost surely with respect to $\mu$. Indeed, we could arrange 
the initial conditions into fixed points of the coupled maps and consider
the particular case that  $h(x,y) = h(x-y)$ and $h(0) = 0$. 
Then the reduction obviously fails in the  realization where $x_1 = \cdots = x_n = \tilde x$. 
In this situation the dynamics of the hubs are given by 
isolated dynamics, and not by our reduction. 
This example shows that one needs to remove  a zero measure set of 
initial conditions

{\it Almost surely every graph:} The assertion that the theorem holds for almost every graph is more subtle. We control the concentration inequalities
of relevant quantities associated with our results. But with small probability close to zero some pathological 
networks may appear. In such exceptional networks the mean field reduction may not hold. We now discuss a particular
example of such a pathology. Recall that the networks we consider are oriented. However, $\mathbb{E}_{\bm{w}}(k_i) = \mathbb{E}_{\bm{w}}(k_i^{out})$, 
which implies that statistically for a given node the number of incoming and outgoing are the same.  
Let us consider a situation $k_1 = 0$ and $k_1^{out} = \Delta$. We present a pictorial representation of such network in Fig. \ref{OrientedStar}. 
Clearly the dynamics of this hub will not be described by our reduction Theorem, as
it does not receive input from its neighbors and therefore acts as an isolated node.
Notice that 
$$
\mbox{Pr}_{\bm{w}}( k_1=0) = \Pi_{j=1}^n (1 - p_{1j})
$$
together with 
$1 - \Delta w_j \rho \le 1 - \Delta^{\gamma}/n $,  which implies that
$$
\mbox{Pr}_{\bm{w}}( k_1=0) \lesssim e^{- \Delta^{\gamma} }.  
$$
This shows that the probability to find such networks converges to zero as $\Delta$ grows.  
\begin{figure}[htbp] 
   \centering
   \includegraphics[width=2in]{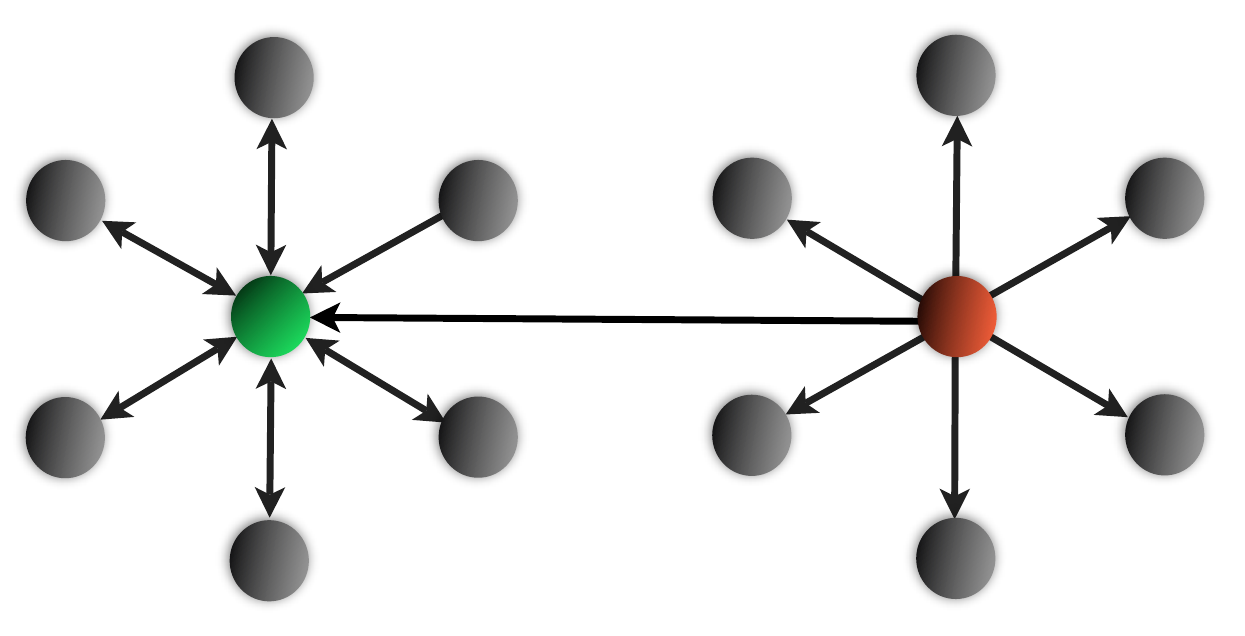} 
   \caption{A network with a fully oriented hub. The hub only exerts influence on its neighbors, but 
   does not receive input from them. Networks with such property form a set of small measure.     
   }
  \label{OrientedStar}
\end{figure}

More generically, due to the concentration inequality described in Proposition \ref{ConIn} the degrees only differ 
from $w_i$ by a amount proportional to $w_i$ in a set of networks of small measure.

\section{Examples of Reductions}
~

We simulate the scenario from the Main Theorem 
using the network model in Example 1 from Section \ref{Nets}, with  $n=2 \times 10^4$, $\ell = 2$ and $\kappa_1 =  1$ and $\kappa_1 \ge \kappa_2 \ge 0 $. We take $w_i = 7$ for $2< i \le 4000$. 
Moreover, we consider $\sigma < 1/2$.  This network can be thought of as composed of a Erd\"os-R\'enyi layer  corresponding 
to $n> i > \ell$ and another layer of highly connected nodes. A pictorial representation of 
such network with  $\ell=2$ and $n=10$ can be seen in Fig. \ref{Illustration}
\begin{figure}[htbp] 
   \centering
   \includegraphics[width=2in]{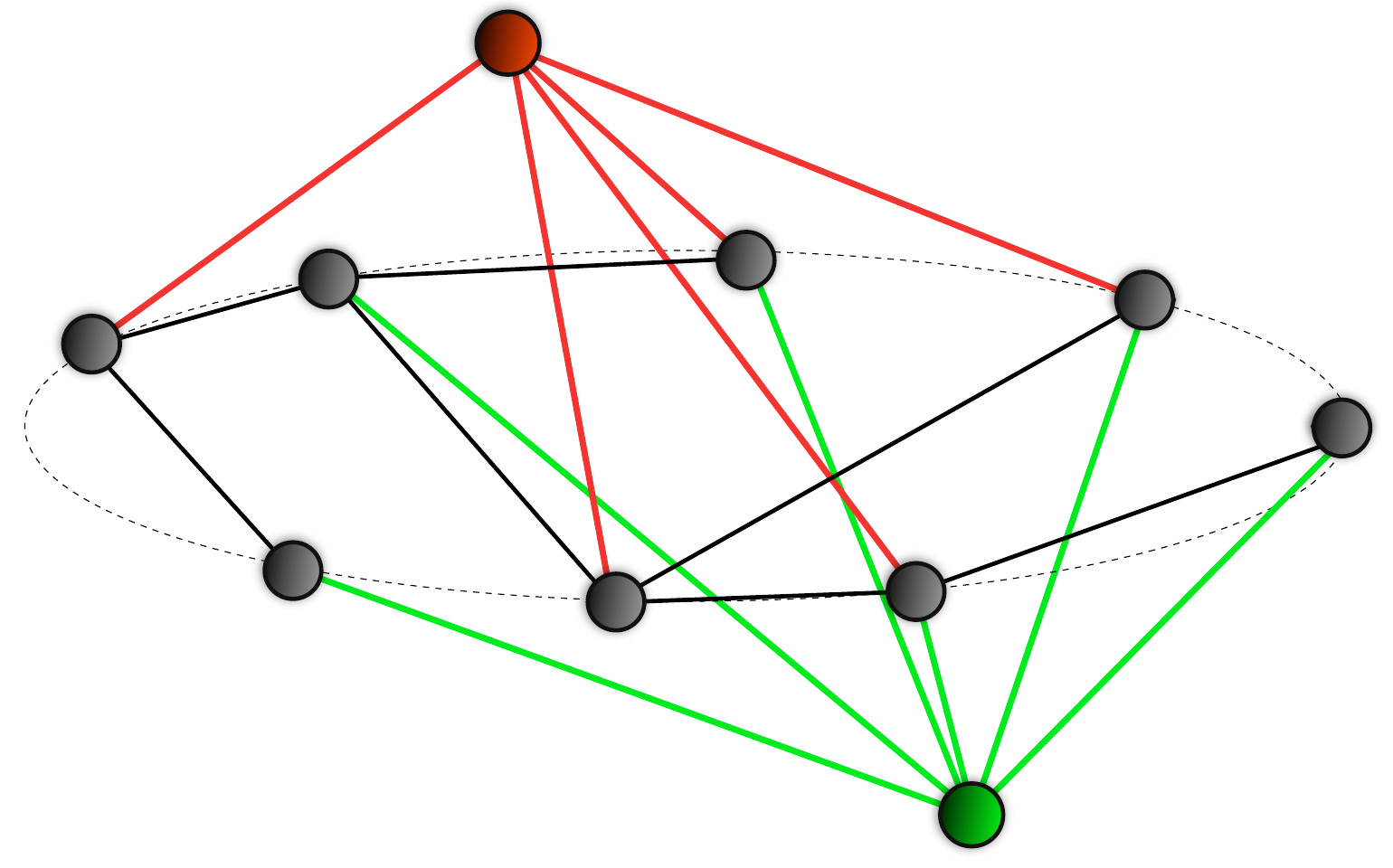} 
   \caption{A pictorial representation of the network described in Example 1 with  $\ell=2$ and $n=10$   }
  \label{Illustration}
\end{figure}

We take the Bernoulli map 
$$
f(x) = 2 x  \, \mbox{   mod} 1
$$
to model the isolated dynamics. This model corresponds to a case with H\"older exponent $\nu=1$.  
The Lebesgue measure $m$ is invariant for the localized system. 
This means that $\mu_0 = 1$. 
Another interesting property of the system is stochastic stability, which means that adding a 
$\delta$-small uncorrelated noise in the evolution the dynamics can still be described by an invariant
measure $\mu_{\delta}$. Moreover, $\mu_{\delta}$ is uniformly close to $m$.  It is easy to see that 
all our results still remain true if this small noise is included. 

Hence, to avoid round-off numerical problems associated with the map $x\mapsto 2x$, 
we introduce a small additive noise $\xi$ with uniformly distributed 
with  support $[0,10^{-5}]$. Hence, the 
isolated dynamics under the influence of this small noise 
is given by 
$$
x(t+1) = 2 x(t) + \xi(t) \, \mbox{   mod} 1
$$
We wish to explore two coupling functions and their consequences for the hub dynamics, 
focusing in particular on the
collective properties of the hubs. To this end, we introduce the following  coherence measure. 
Given points  $x_1(t), x_2(t) \in M$,  with $t \in \{ 1, \dots ,T \}$, we define
the {\em coherence} $r$ between the hubs by
\begin{equation}\label{r}
r e^{i \psi}= \frac{1}{ T } \sum_{j=1}^T e ^{2\pi(  x_1(t)  - x_2(t))}.
\end{equation}

\subsection{Hubs decouple from the network.} 
Consider the coupling function 
\begin{equation}
h (y-x) = \sin 2\pi (y-x).
\label{eq:sin}
\end{equation}
Our reduction technique  renders the following interaction function for the hubs
$$
g(x) = \int_0^{1} h (y-x) dy = 0.
$$ 
According to the Main Theorem, the coupling equation
(\ref{eq:sin}) gives the reduced equations 
\begin{equation}
x_i(t+1) = 2 x_i(t) + \alpha \zeta_i   \mbox{ ~ mod 1}, \, \mbox{ for }i=1,2 \mbox { ( $\ell=2$) }
\end{equation}
 and where $\zeta_i$ is the noise term. 
Hence, the hubs effectively decouple from the network, as there is no interaction with the hubs with the mean field. 
The only effect of the network  on the hubs is a noise-like term $\zeta_i$
and the parameter $\alpha$ only appears in the noise term.
The stochastic stability of the Bernoulli family implies that for $T$ large enough $r \approx 0$.

First, for a fixed $\Delta=260$ and $\kappa_2=0.99$ we compute $r$ as a function of $\alpha$. In the computation 
of the coherence measure $r$ we discard the first $10^3$ iterates and consider $T = 10^3$.
As predicted no coherence is attained, and the behavior of $r$ as a function of $\alpha$ is flat.  
The results are presented in Fig. \ref{Alpha0}a).  In   Fig. \ref{Alpha0}b)  the time series of 
$| \cos 2\pi x_1(t)  - \cos 2\pi x_2(t) |$ are shown.  

\begin{figure}[htbp] 
   \centering
   \includegraphics[width=5in]{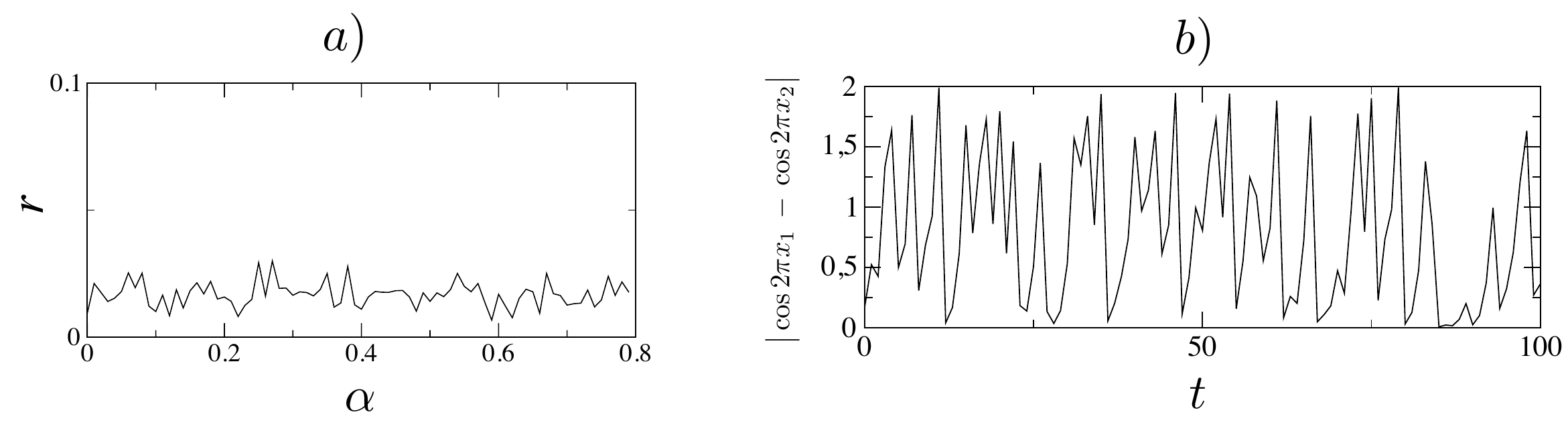} 
   \caption{No coherent dynamics between hubs. In a) we show the coherence measure $r$ as a function of alpha. In b) 
    the time series  of  $| \cos 2\pi x_1(t)  - \cos 2\pi x_2(t) |$.  
   }
  \label{Alpha0}
\end{figure}

\subsubsection{Effects of $\Delta$ on the fluctuations}

We now perform a set of simulations to study the scaling relations between $\zeta$ and $\Delta$. 
Notice that in the model of  Example 1 we take 
$\gamma=1$, hence our Theorem predicts a scaling as
$\zeta \lesssim \Delta^{-1/2}$. 

In this set of experiments we vary $\Delta$, recall that  $n=2 \times 10^4$ and the low degrees $w_i=7$ for $2 < i \le n$ are fixed. 
We also fix $\alpha = 0.1$, and consider only the fluctuations $\zeta_1$ on the main hub $x_1$. For simplicity of notation we shall write
$\zeta = \zeta_1$.    For each experiment we compute the quantity
$$
\langle | \zeta | \rangle  = \frac{1}{T} \sum_{t=1}^T | \zeta(t) |
$$
where $T=10^3$. This mean value of the modulus of the fluctuations $\zeta$ must have the same 
scaling relation $\langle | \zeta | \rangle \lesssim \Delta^{-1/2}$. For each $\Delta$ we construct the 
corresponding network only once and measure $\langle | \zeta | \rangle$. This 
implies that we do not
average $\langle | \zeta | \rangle$ over the network ensemble, as the networks have good concentration 
properties. The simulation results are presented
in Fig. \ref{ZetaDelta} and are in agreement with the predictions
\begin{figure}[htbp] 
   \centering
   \includegraphics[width=2.2in]{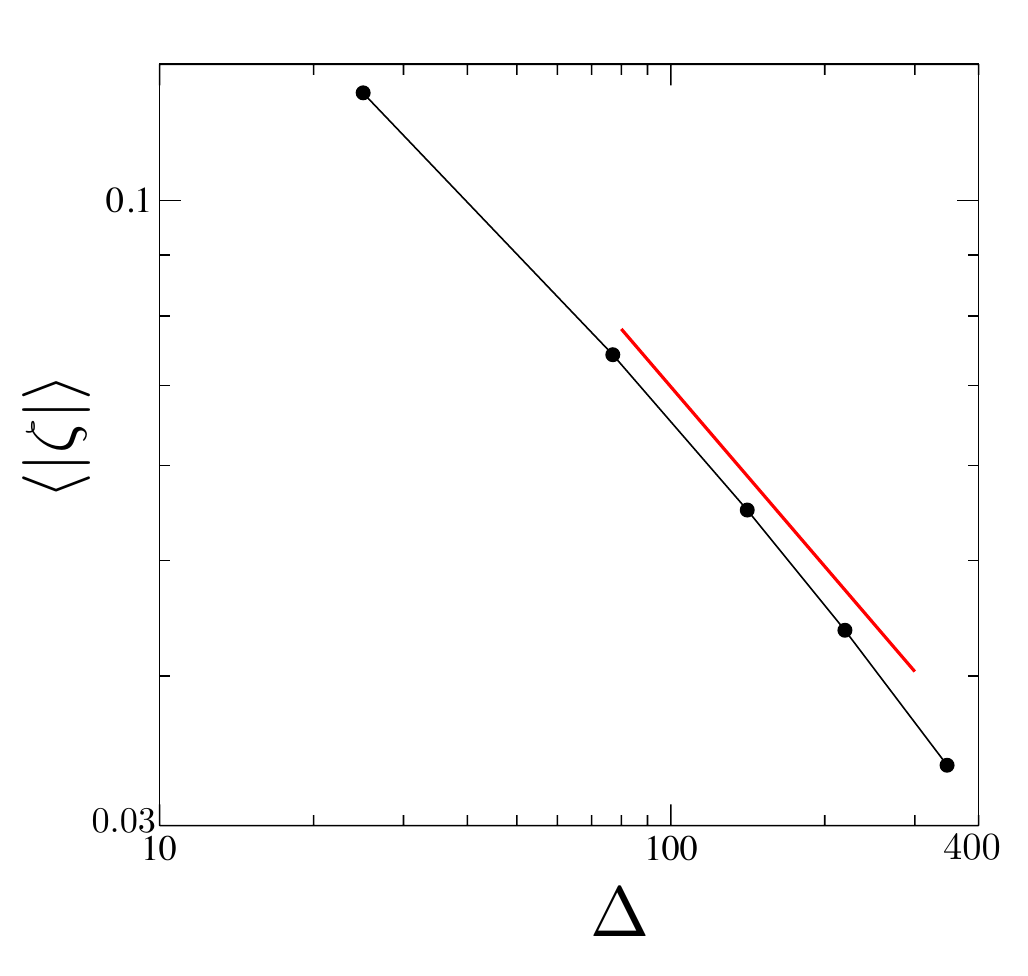} 
   \caption{Scaling relation between the mean value of the modulus of the fluctuations $\langle | \zeta | \rangle$ and the 
   expected degree $\Delta$. The simulation results are shown in black full points and a curve with 
   scaling $\Delta^{-1/2}$ is shown in red (full bold line). Our Theorem predicts a scaling relation 
   $\langle | \zeta | \rangle \lesssim \Delta^{-1/2}$, which is in agreement with the simulation results. 
   }
  \label{ZetaDelta}
\end{figure}

\subsubsection{Effects of the normalized degree $\kappa$ on the fluctuations}

As before $n$  and the $w_i$'s for the low degree nodes are fixed. Now  we also fix  $\Delta = 347$  and vary $\kappa_2$ to study the scaling relations between $\zeta$ and $\kappa$. Our Theorem predicts a scaling as $\zeta \lesssim \kappa^{1/2}$. We then compute the mean value of the modulus of the fluctuations $\langle | \zeta_2 | \rangle$.  In this subsection, for simplicity of notation we shall write  $\langle | \zeta | \rangle = \langle | \zeta_2 | \rangle$ and $\kappa = \kappa_2$.  Again, for each $\kappa$ we construct the corresponding network only once and measure $\langle | \zeta | \rangle$. The simulation results are presented in Fig. \ref{ZetaKappa} and is in agreement with the predictions
\begin{figure}[htbp] 
\centering
\includegraphics[width=2.5in]{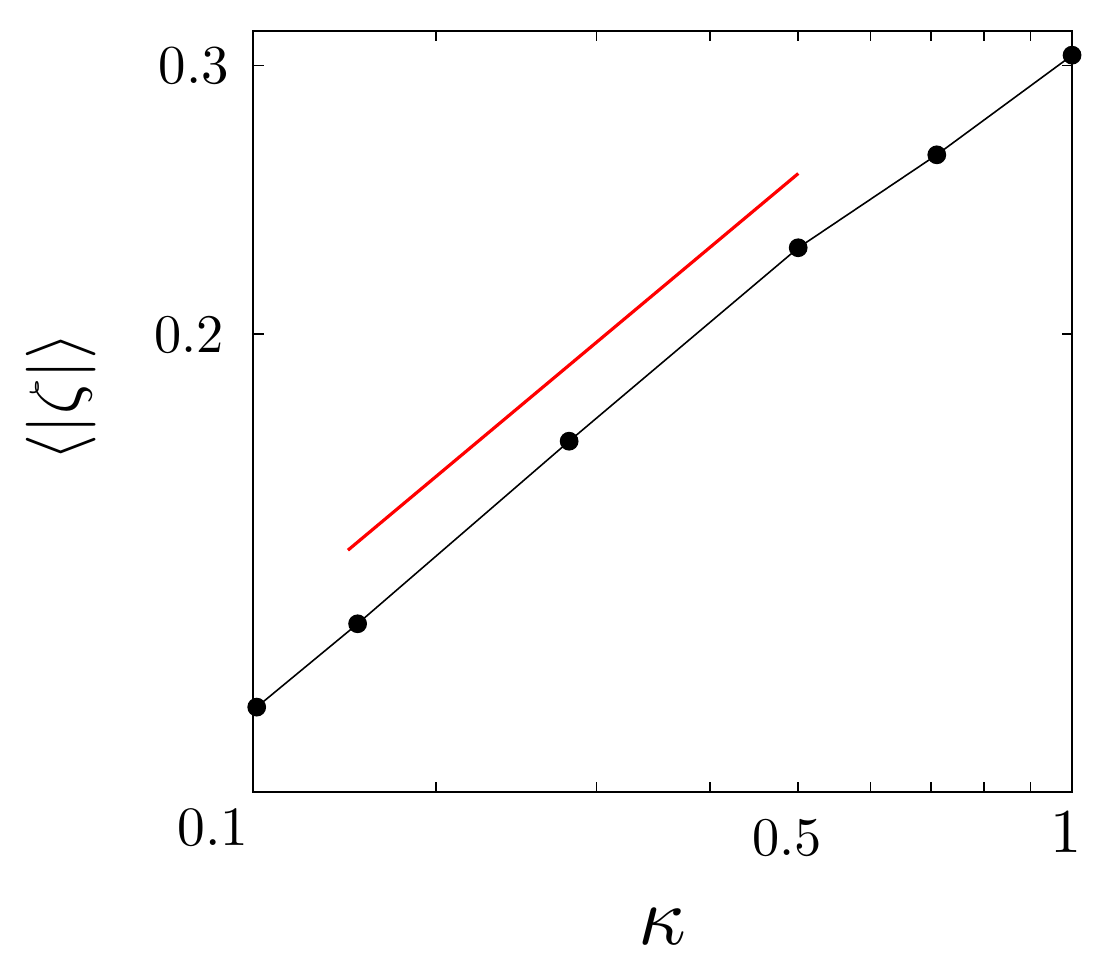} 
  \caption{Scaling relation between the mean value of the modulus of the fluctuations $\langle | \zeta | \rangle$ and   	  normalized  degree $\kappa$. The simulation results are shown in black full points and a curve with    scaling $\kappa^{1/2}$ is shown in red (full bold line). For fixed network parameters $n,\Delta$  Our Theorem predicts a scaling relation $\langle | \zeta | \rangle \lesssim \kappa^{1/2}$, which is in agreement with the simulation results. }
\label{ZetaKappa}
\end{figure}

\subsection{Reduction Reveals Coherent Behavior}  

The reduction reveals that coherent dynamics
can be caused by cancellations due to the nature of the
coupling functions and the dynamics of the isolated 
dynamics (in terms of the invariant measure). 
Depending on the coupling function the hubs may exhibit a coherent behavior for a range of 
coupling strengths $\alpha$. 
Here, we illustrate such scenario. Again, we consider the network from Example 1, with $n=2 
\times 10^3$, $w_i = 7$ for $2< i \le n$. We fix
$\Delta = 347$, with $\kappa_1 = 1$ an $\kappa_2 = 1$.
Then we construct one realization of such a network. The coupling function 
\begin{equation}
h(x,y) = \sin 2 \pi y - \sin 2 \pi x.
\label{eq:sinsin}
\end{equation}

We performed extensive numerical simulations to compute the coherence measure for the two hubs $x_1 $ and $x_2$ as a function of the coupling strength $\alpha$.  The result can be seen in Fig. \ref{r}. The parameter $r$, see (\ref{r}) has an intricate dependence on $\alpha$. Such behavior can be 
uncovered by our reduction techniques. 
\begin{figure}[htbp] 
   \centering
   \includegraphics[width=2.5in]{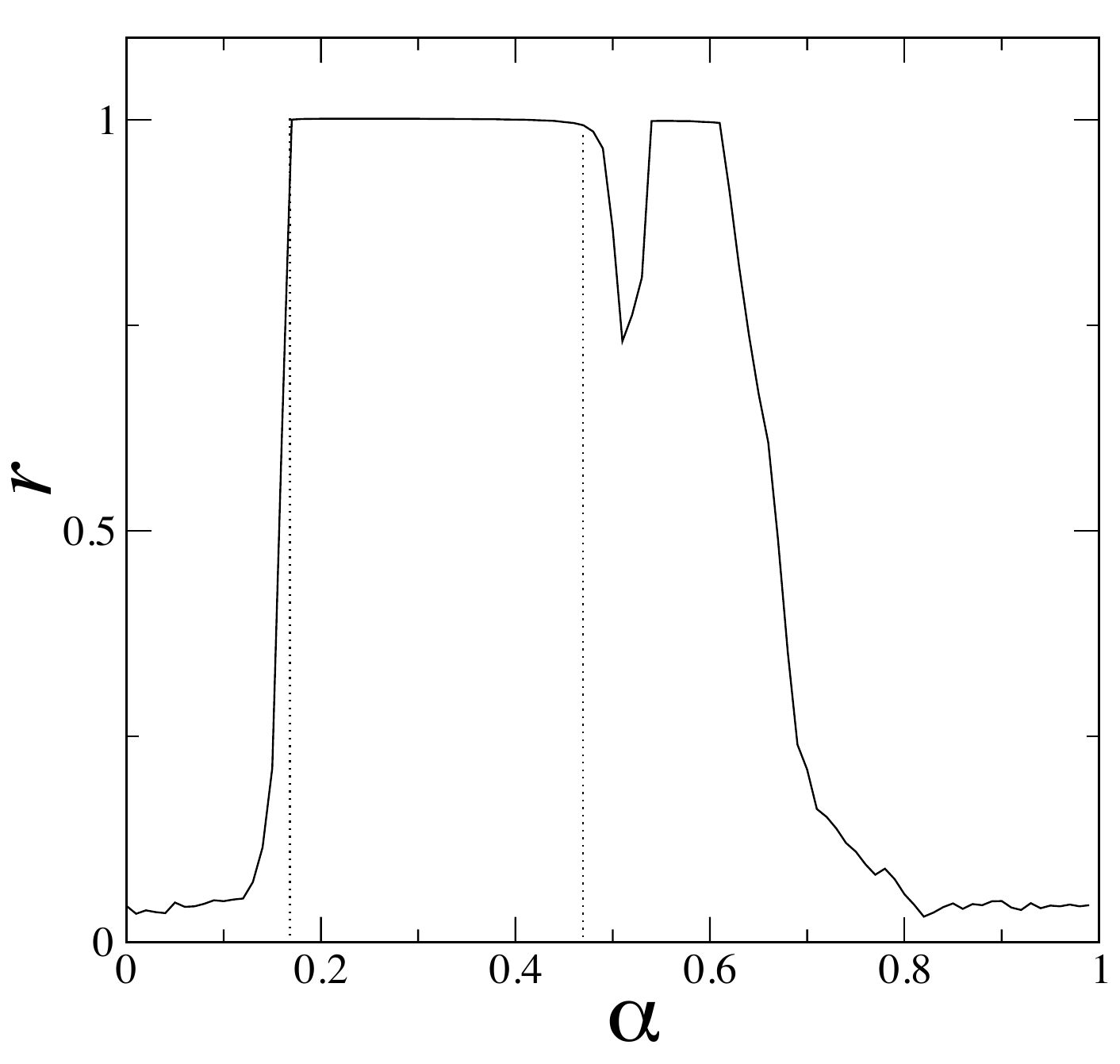}
   \caption{The coherence parameter $r$ versus the coupling strength $\alpha$. The high values of 
   $r$ corresponds to regimes where the dynamical of the hubs $x_1$ and $x_2$ are correlated. The onset of coherent
   regimes can be predicted by the reduction techniques, e.g., the large coherent plateau is given by (\ref{alphaB}). \label{fig:r}}
\end{figure}

Applying the Main Theorem with the coupling (\ref{eq:sinsin}) yields 
the following effective equations for the  hubs dynamics 
\begin{equation}
{x}_i(t+1) = 2  x _i(t)) + \alpha  \sin 2 \pi x_i + \alpha{\zeta}_i(t) \mbox{ mod 1 },
\label{eq:sinsinco}
\end{equation}
Let us neglect the finite size fluctuations for a moment. Then the dynamics of the hubs are described by the 
following equation
\begin{equation}\label{EsZ}
{x}(t+1) = 2 x(t) + \alpha \sin 2 \pi x   \mbox{ ~  mod 1},
\end{equation}
and so the parameter $\alpha$ determines the dynamics significantly. 
Indeed (\ref{EsZ})  has a trivial fixed point $x=0$ (identified with $1$) for all $\alpha$,
but  this fixed point is only stable for 
\begin{equation}
\frac{1}{2 \pi } < \alpha  < \frac{3}{2 \pi}.
\label{alphaB}
\end{equation}
In this range if the initial conditions for the hubs start in a vicinity of $0$ then they 
will remain there for all future times. 
This will correspond to a trivial coherent behavior. 
Hence, while the low degree behaves in an erratic fashion the hubs, although isolated 
chaotically, will stay in a steady state. This regime corresponds to large  values of 
the coherence measure $r$. On the other hands, when $\alpha\notin (\frac{1}{2 \pi },  \frac{3}{2 \pi})$ this fixed point becomes repelling and indeed
the second plateau in Fig. \ref{fig:r} corresponds to a stable periodic orbit of period two. 
Examples of such dynamics can be observed in Fig. \ref{AlphaG}

\begin{figure}[htbp] 
   \centering
   \includegraphics[width=4in]{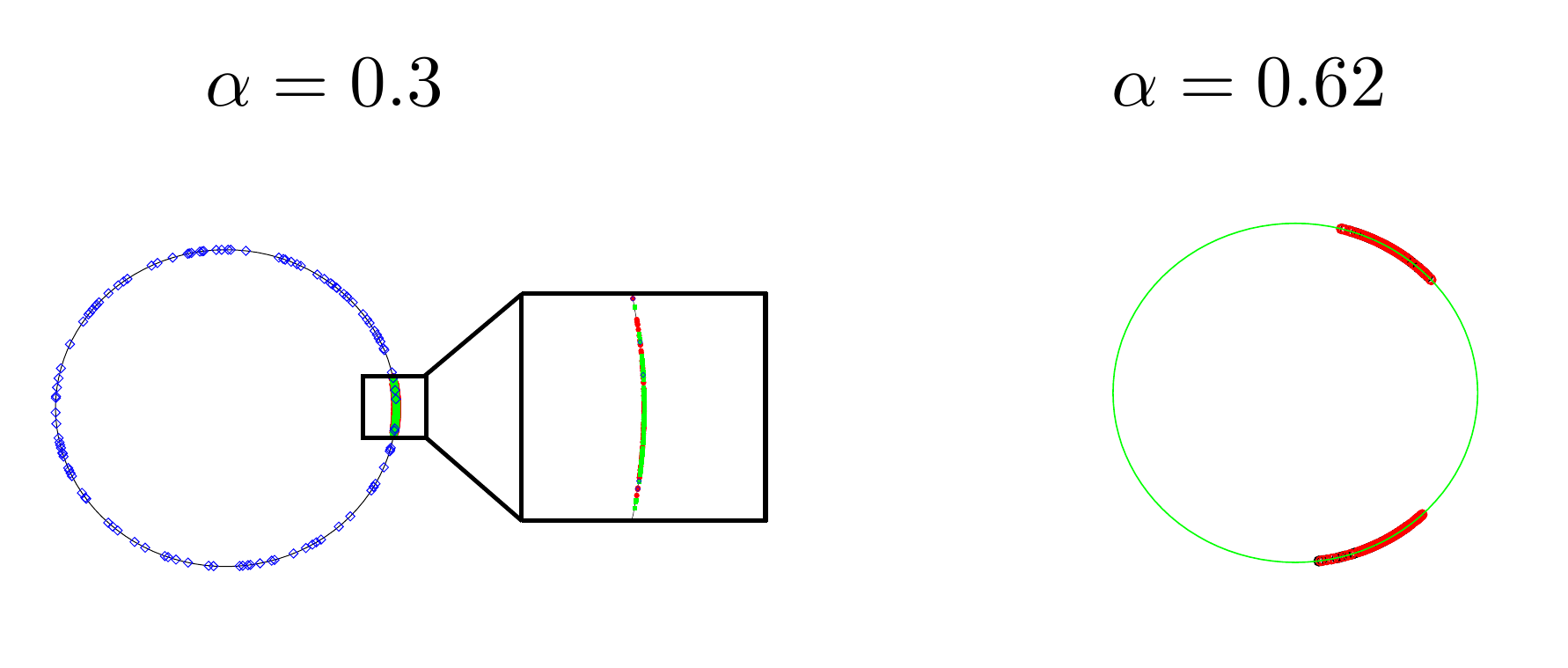} 
   \caption{ Coherent dynamics of the hubs. For $\alpha=0.3 $ we depict the trajectories of the hubs
   $x_1$ (solid circles) and $x_2$ (solid squares) along with the trajectories of the low degree node $x_{400}$ (diamonds). The trajectories are shown 
   in the circle. The trajectories of the hubs stay 
   close to an equilibrium point, whereas the trajectories of the low degree nodes spread over the whole circle. For $\alpha = 0.62$ we depict the trajectories of the hubs
   in the circle. They display a coherent periodic motion of period two.}
  \label{AlphaG}
\end{figure}

\section{Dynamics of Low degree Nodes}

\subsection{Isolated Nodes}\label{IsoN}
The proof of the Main Theorem is based on the ergodic properties of $f$ and is derived from the 
properties of the Peron-Frobenius operator (acting on some convenient Banach space).
To prove that this operator has fixed points, 
we use standard technique based on the notion 
of the projective metric associated with a convex cone in a vector space. 
We follow closely the exposition in Ref. \cite{Viana}. We use the distance $d$ induced by the riemannian metric. 

Let $E = C^0(M,\mathbb{R})$ be the space of continuous real valued functions 
defined on $M$. Let $C = C(a,\nu)$ be the convex cone of functions  $\varphi \in E$ such that
\begin{itemize}
\item[i)] $\varphi(x) >0$ for all $x\in M$ and 
\item[ii)] $\log \varphi$ is $(a,\nu)$-H\"older continuous on $\rho_0$ neighborhoods.
\end{itemize}
This last condition means that for all $x,y \in M$ such that $d(x,y) \le \rho_0$  then 
$$
e^{-ad(x,y)^{\nu}} \le \frac{\varphi(x)}{\varphi(y)} \le e^{a d(x,y)^{\nu}}.
$$

The projective metric (Hilbert metric) is introduced as follows given $\varphi_1, \varphi_2 \in C$, define
$$
\alpha (\varphi_1 ,\varphi_2 ) = \sup\{ t > 0 : \varphi_2 - t \varphi_1 \in C\},
$$
and
$$
\beta(\varphi_1 ,\varphi_2) = \inf\{ s > 0 : s \varphi_1 - \varphi_2 \in C\}.
$$
Then, 
\begin{equation}\label{theta}
\theta(\varphi_1 , \varphi_2 ) = \log \frac{\beta(\varphi_1, \varphi_2)}{\alpha(\varphi_1,\varphi_2)}
\end{equation}
is a metric in the projective quotient of $C$. 

Statistical properties of the isolated dynamics can be obtained
by means of the transfer operator $\mathcal{L}$. Let 
$\varphi : M \rightarrow \mathbb{R}$, the operator $\mathcal{L}$ is defined by 
$$
(\mathcal{L} \varphi)(y) = \sum_{f(x) = y} \frac{\varphi(x)}{ | \mbox{det} Df (x)| }
$$

Given a measure $\mu$ absolutely continuous 
with respect to the Lebesgue measure $m$, that is, given $A \subset M$ one has 
$\mu(A) = \int_A \varphi dm $, for an integrable function $\varphi$, then the pushforward of $\mu$ 
under $f$ has the duality property 
\begin{eqnarray}
(f_* \mu)(A) 
&=&\int_A \mathcal{L} \varphi dm.
\end{eqnarray}
\noindent
The fixed points $\varphi_0$ of the linear operator $\mathcal{L}$ are $f$-{\em invariant}
absolutely continuous probabilities measures. Conversely, if a $f$-invariant
probability $\mu_0$ is absolutely continuous with respect to $m$ then 
there exists the Radon-Nikodym derivative $d \mu_0 /dm = \varphi_0$ and 
$\mathcal{L} \varphi_0 = \varphi_0$.

The domain of $\mathcal{L}$ plays a vital role. Another important point is that the metric space
$(C(a,\nu),\theta)$ is not complete. Nonetheless, $\mathcal{L}$
acting on $C(a,\nu)$ is a contraction with a unique fixed point. 

\begin{propo}\label{Lfp}
The operator $\mathcal{L} : C(a,\nu) \rightarrow C(a,\nu)$ is a contraction with respect 
to the projective metric $\theta = \theta_{a,\nu}$ associated with the convex cone $C(a,\nu)$.
Moreover, $\mathcal{L}$ has a unique fixed point $\varphi_0 \in C(a,\nu)$. 
\end{propo}

\begin{proof}See Propositions 2.2, 2.4 and 2.6 cf. \cite{Viana}. 
\end{proof}

\subsection{Transfer Operator of the Low degree Nodes}

We study the dynamics of the low degree nodes as perturbations 
of the isolated dynamics, and then obtain the statistical properties of the low degree
in terms of perturbation results of the transfer operator.

To this end, we show that for almost every network, low degree nodes can be viewed as a  perturbation. 
Notice that 
\begin{equation}
{x}_i(t+1) = { f}({ x}_i(t)) + \frac{\alpha}{\Delta} \sum_{j=1}^n A_{ij} h(x_j(t), x_i(t)),\,\,  i=1,\dots,n
\label{md1a}
\end{equation}
can be rewritten 
as 
\begin{equation}
{x}_i(t+1) = { f}({ x}_i(t)) + \alpha r_i(x(t),y_i(t)),\,\,  i=1,\dots,n
\label{md1b}
\end{equation}
where the coupling term taking into account  (\ref{h}) is represented as
$$
r_i(x_i,y_i) = \sum_{p,q}  u_p(x_i) y_{q,i}  
$$
with
$$
y_{q,i} =  \frac{1}{\Delta} \sum_j A_{ij}  v_q(x_j). 
$$

Our next result guarantees that the coupling term can be made uniformly small as $\Delta$ increases.

\begin{propo}\label{rDelta}
Let $\mathcal{G}(\bm{w})$ be strongly heterogeneous.  Then  the coupling term 
$r_i$ viewed as a mapping $x \mapsto r_i(x,y_i)$ is $C^{1+ \nu}$, and  for every low degree node $\ell < i \le n$ 
the coupling term satisfies 
$$
\| r_i (x_i, y_i ) \| \lesssim \Delta^{-\gamma} 
$$
for almost every network in $\mathcal{G}(\bm{w})$.
\end{propo}

\begin{proof} The claim on the regularity of $r_i$ is trivial. The second claim follows from the concentration properties. Since the 
manifold is compact, we obtain for every $x\in M$
$\left\| v_q(x) \right\| \le K $ and $\left\| u_p(x) \right\| \le K 
$
Now we wish to show concentration properties 
for $y_{q,i}(t)$. Hence, for a given fixed $t$ for simplicity we denote $y_{q,i} (t)= y_{q,i}$ and estimate
\begin{equation}\label{Ey}
\mathbb{E}_{\bm{w}} (y_{q,i})  = \frac{w_i}{\Delta} \rho \sum_{i=1}^n w_j v_q( x_j )
\end{equation}
leading to 
$$
\|  \mathbb{E}_{\bm{w}} (y_{q,i})  \|  \lesssim \Delta^{- \gamma}.
$$
To obtain our claim we also estimate the variance
$$
\mbox{Var}_{\bm{w}}(y_{q,i}) = \mathbb{E}_{\bm{w}}(y_{q,i}^2) - \mathbb{E}_{\bm{w}}^2 (y_{q,i}),
$$
hence, we need to estimate 
\begin{eqnarray}
\mathbb{E}_{\bm{w}}(y_{q,i}^2) &=& \mathbb{E}_{\bm{w}} \left(  \frac{1}{\Delta^2}  \sum_{j,k} A_{ij}A_{ik} v_q( x_j ) v_q(  x_k  ) \right) \\ 
&=& \mathbb{E}_{\bm{w}} \left(  \frac{1}{\Delta^2}  \sum_{j} A^2_{ij}v_q( x_j) ^2  + \frac{1}{\Delta^2}  \sum_{j\not=k} A_{ij} A_{ik} v_q( x_j ) v_q(  x_k ) \right) .
\end{eqnarray}
But $A_{ij}^2 = A_{ij}$, hence, 
\begin{eqnarray}
\mathbb{E}_{\bm{w}}(y_{q,i}^2) &=& \left(  \frac{w_i}{\Delta^2}  \rho \sum_{j} w_j v( x_j ) ^2  + \frac{w_i^2}{\Delta^2}  \rho^2\sum_{j\not=k} w_j w_{k} v_q( x_j ) 
v_q( x_k ) \right) .
\end{eqnarray}

Combining both computations for  $\mathbb{E}_{\bm{w}}(y_{q,i}^2)$ and  $\mathbb{E}^2_{\bm{w}}(y_{q,i})$
we obtain
$$
\mbox{Var}_{\bm{w}}(y_{q,i}) = \frac{w_i}{\Delta^2} \rho \sum_{j} w_j v_q( x_j) ^2  -  \frac{w_i^2}{\Delta^2}  \rho^2 \sum_{j} w_j^2 v_q( x_j) ^2, 
$$
and as the functions $v_q$ are bounded
$$
\| \mbox{Var}_{\bm{w}}(y_{q,i}) \|  \le  K^2\left(  \frac{w_i}{\Delta^2}  +  \frac{w_i^2}{\Delta^2}  \rho^2 \sum_{j} w_j^2 \right).
$$

Recall that by property N2 of the strong heterogeneity $w_i \le \Delta^{1-\gamma}$, and note that $\sum_j w_i^2 \le \Delta / \rho $ which leads to the 
following estimate
$$
 \frac{w_i^2}{\Delta^2}  \rho^2 \sum_{j} w_j^2 \le \frac{w_i^2 \rho}{\Delta}  \lesssim \Delta^{-1 - 3 \gamma}. 
$$
This leads to 
\begin{eqnarray}
\| \mbox{Var}_{\bm{w}}(y_{q,i}) \| &\lesssim& \Delta^{-1 - \gamma} +\Delta^{-1 -3\gamma}  ~ \lesssim ~ \Delta^{-1 - \gamma}.
\end{eqnarray}
Hence, for every $(1-\gamma) /2 > \beta >0$ we obtain
$$
\mbox{Pr}_{\bm{w}} \left( \| y_{q,i} - \mathbb{E}_{\bm{w}}(y_{q,i}) \| \gtrsim \Delta^{-(1+\gamma)/2 +\beta}  \right) \lesssim 
\Delta^{-2 \beta}$$ 
implying that for almost every network in the model $\mathcal{G}(\bm{w})$,
$$
\| y_{q,i} -  \mathbb{E}_{\bm{w}}(y_{q,i}) \| \lesssim \Delta^{-(1+\gamma)/2 + \beta}.
$$
By the triangle inequality $| \|  y_{q,i} \|  - \| \mathbb{E}_{\bm{w}}(y_{q,i}) \| | \lesssim \Delta^{-(1+\gamma)/2 + \beta}$
 together with 
 (\ref{Ey}) and taking $\beta$ small enough we obtain that 
$$
\| y_{q,i} \| \lesssim \Delta^{-\gamma}\, .
$$
Now recall that the function $u_p$ are uniformly bounded over $M$  yielding
$$
\| r_i(x_i,y_i) \| = \| \sum_{p,q}  u_p(x_i) y_{q,i}  \| \lesssim \Delta^{-\gamma}
$$
for almost every network as $y_{q,i}$ has the derived concentration properties.
\end{proof}

This result reveals that for almost every network $\mathcal{G}(\bm{w})$ the network effect
on low degree nodes is a perturbation in the limit of large $\Delta$.  We use this remark to treat the 
mean field reduction in terms of the ergodic properties of the perturbed maps.  

Next we consider the index $i$ fixed on a low degree node. The following argument will hold for any 
low degree node. We can view the perturbed map $f_t = f + r$ as a random-like perturbation of the 
map $f$ by writing
$$
f_t(x) = f(x) + r(x(t),t)
$$ 
where $r$ is small by Proposition \ref{rDelta}.
Note that the maps $f_t$ are uniformly close in the $C^1$ topology to $f$. 
We view $f_t$ as parametrized families $f_t : M \rightarrow M $ of $C^{1+\nu}$-H\"older continuous maps.  So at each time step in its evolution, we pick a 
map $f_t$ in an open neighborhood of $f$. 
Denote $\bm{t} = (t_1, t_2, \cdots)$ and define
$$
f_{\bm{t}}^k = f_{t_k} \circ f_{t_{k-1}} \circ \cdots \circ f_{t_1},
$$
hence, the equation may be recast as
$$
x^{k+1} = f_{\bm{t}}^k (x_0).
$$
For that map we introduce perturbed versions of the linear operator $\mathcal{L}$ 
$$
(\mathcal{L}_t \varphi)(y)  = \sum_{f_t(x)=y} \frac{\varphi(x)}{|\mbox{det}Df_t (x)|}.
$$
We now claim the following  
 \begin{lemma}[Perturbation]\label{Norm} 
Consider the transfer operator of both $f$
and $f_t$ acting on the space  $C(a,\nu)$. Then 
$$
\| \mathcal{L}\varphi - \mathcal{L}_t \varphi  \| \lesssim \Delta^{- \gamma \nu}
$$
on the norm of uniform convergence.
 \end{lemma}

\begin{proof} Note that the main difference here to the stochastic stability analysis performed in \cite{Viana} 
is that the maps $f_{t}$ are not chosen independently, which may lead to the non-existence
of stationary measures. Since 
$f$ is a local diffeomorphism, all the points $y\in M$ have a same number $k\ge 1$ of
preimages, namely the degree of $f$. (In our case $k=2$.)
Moreover, given any preimage $x$ of $y$, 
there exists a neighbourhood $V$ of $y$ and an inverse branch $g : V \rightarrow M$
such that $f \circ \phi = identity$ and $\phi(y) =x$. A local 
inverse branch $g$ must be contracting 
$$
d(\phi(y),\phi(y^{\prime})) \le \sigma^{-1} d(y,y^{\prime}) 
$$
for every $y,y^{\prime}$ in $V$. Moreover, by compactness of $M$, there exists $\rho_0$ such that given $y_1$, $y_2 \in M$ with $d(y_1,y_2) \le \rho_0$,
one may write $f^{-1}(y_j) = \{ x_{j1}, \cdots , x_{jk} \}$  for $j=1,2$ with 
$$
d(x_{1i}, x_{2i}) \le \sigma^{-1} d(y_1,y_2)
$$
for each $i=1,\cdots,k$. 
Obviously
$$
f^{-1}_t(y) = \{  x_{1,t} , \cdots , x_{k,t}  \}.
$$
With these remarks, we can write the transfer operators as 
$$
(\mathcal{L} \varphi)(y) = \sum_{i=1}^k \frac{\varphi(x_i)}{|\mbox{det}Df(x_i)|}  \, \, \mbox{  ~  ~  and   ~ ~  } \, \,
(\mathcal{L}_t \varphi)(y) = \sum_{i=1}^k \frac{\varphi(x_{i,t})}{|\mbox{det}Df_t(x_{i,t})|} .
$$

We proceed by obtaining bounds for the transfer operator $\mathcal{L}_t$. First notice that
as $f_t$ is $C^{1+ \nu }$ and uniformly $\Delta^{-\gamma}$ close to $f$ 
\begin{equation}\label{xit}
\sup\{ d ( x_{i,t}, x_i) : 1 \le i \le k,  y\in M  \}  \lesssim {\Delta}^{-\gamma} .
\end{equation}

To obtain bounds on the Jacobian we first note that
\begin{equation}\label{Df_Dif}
\mbox{det} Df_t(x_{i,t}) - \mbox{det} Df(x_i) =     \mbox{det} Df_t(x_i) - \mbox{det} Df(x_i)  + \mbox{det} Df_t(x_{i,t})
- \mbox{det} Df_t(x_i). 
\end{equation}
We estimate the first difference in the right hand side of (\ref{Df_Dif}) as follows. 
Since $f_t = f + r$ is a uniformly $\Delta^{-\gamma}$ close to $f$. Moreover, 
\begin{eqnarray}
| \mbox{det} Df(x_i) - \mbox{det} Df_t(x_i) |  &=& |\mbox{det} Df(x_i)| \left| 1 - \mbox{det} \left[ Df_t(x_i) Df^{-1}(x_i) \right] \right| \nonumber \\
&=& |\mbox{det} Df(x_i)| \left| 1 - \mbox{det} \left[ D(f+r)(x_i) Df^{-1}(x_i) \right] \right| \nonumber \\
&=& |\mbox{det} Df(x_i)| \left| 1 - \mbox{det} \left[ \mbox{I} + Dr(x_i) Df^{-1}(x_i) \right] \right|,  \nonumber 
\end{eqnarray}
and
$$
\mbox{det} \left[ \mbox{I} + Dr(x_i) Df^{-1}(x_i) \right] = 1 + \mbox{tr} D r(x_i) Df^{-1}(x_i).
$$
Noting that expansivity implies that the derivative
$Df$ is an isomorphism at the every point, and in view of Proposition \ref{rDelta} we obtain 
$$
\left| \mbox{tr} \left[ Dr(x_i) Df^{-1}(x_i) \right] \right| \lesssim \Delta^{-\gamma},
$$
and therefore
\begin{equation}\label{D1}
| \mbox{det} Df(x_i) - \mbox{det} Df_t(x_i) | \lesssim \Delta^{-\gamma}\, .
\end{equation}

To control the second difference in the right hand side of  (\ref{Df_Dif}) we note that $Df_t$ is 
$(a,\nu)-$H\"older continuous implying that the determinant is $\nu-$H\"older continuous
$$
|  \mbox{det} Df_t(x_{it}) - \mbox{det} Df_t(x_i) | < b d( x_{i,t}, x_{i})^{\nu},
$$
yielding in view of  (\ref{xit})
\begin{equation}\label{D2}
|  \mbox{det} Df_t(x_{it}) - \mbox{det} Df_t(x_i) | \lesssim \Delta^{-\gamma \nu}.
\end{equation}

These estimates  (\ref{D1}) and  (\ref{D2}) for the right hand side terms of  (\ref{Df_Dif}) yield the following estimate
\begin{equation}\label{Df}
| \mbox{det} Df_t(x_{it})  -  \mbox{det} Df(x_i)  | \lesssim  \Delta^{-\gamma \nu}. 
\end{equation}

To obtain the result, it remains to control the test functions $\varphi$. 
Note that since $\log \varphi$ is $(a,\nu)$-H\"older continuous we have
$$
| \varphi(x_{i,t})  - \varphi(x_{i})| \le \varphi(x_{i}) e^{a d( x_{i,t}, x_{i})^{\nu}}
$$ 
and since the manifold is compact $\varphi$ will attain its maximum on $M$. Hence,
\begin{equation}\label{DifVa}
| \varphi(x_{i,t})   -  \varphi(x_{i}) | \lesssim \Delta^{- \gamma \nu}
\end{equation}
as $1 - e^{a d( x_{i,t}, x_{i})^{\nu}} \le d( x_{i,t}, x_{i})^{\nu}$, once,  $a d( x_{i,t}, x_{i})^{\nu} \le 1$ as $\Delta$ is large.
Altogether the estimates  (\ref{Df}) and  (\ref{DifVa}) imply that for all $y \in M$
$$
\| \mathcal{L}_t \varphi - \mathcal{L} \varphi \|_0 \lesssim \Delta^{-\gamma  \nu}, 
$$
concluding the result. 
\end{proof}

The above proposition reveals that at each step the transfer operator of the low degree nodes
is uniformly close to the unperturbed operator. The natural question concerns the behavior of 
compositions of the transfer operators, corresponding to the time evolution. 
Hence, we introduce the transfer operator
$$
\mathcal{L}^k_{\bm{t}} = \mathcal{L}_{t_k} \circ \mathcal{L}_{t_{k-1}} \circ \cdots \circ \mathcal{L}_{t_1} .
$$

Our next result characterizes the action of $\mathcal{L}^k_{\bm{t}}$.

\begin{propo}\label{Near}
Let $\varphi \in C(a,\nu)$ then there exists $u = u(\varphi)>0$ such that  almost every network 
$$
\mathcal{L}^k_{\bm{t}} \varphi=  \varphi_0 + \tilde{\mu}_k 
$$
for all $k > u$ where $\varphi_0\in C(a,\nu)$ is the fixed point of the unperturbed operator $\mathcal{L}$ and $\tilde{\mu}_k\in C(a,\nu)$, $k=1,2,\dots$  satisfy
$$
\left\| \tilde{\mu}_k\right\|_0 \lesssim \Delta^{-\gamma \nu}  \mbox{~ for all ~} k > u.
$$
\end{propo}

Before we prove this Proposition, we need the following auxiliary result concerning the behavior of a family
of operators near a contraction. Given a metric space $(X,d)$, we denote the open ball 
$$
B(z,\delta) := \{x \in {X} : d( x, z) < \delta \}.
$$
Our claim is as follows

\begin{lemma}\label{ContrF}
Let $(X,d)$ be a metric space and transformation $F : {X} \rightarrow {X}$ with Lipschitz constant $k<1$ and a unique fixed point $z$. 
Let  $T$ be a metric space and consider a family of transformations $F_t : {X} \rightarrow {X}$, with $t \in T$, 
and introduce the composition 
$$
F^n_{\bm{t}} = F_{t_n}  \circ F_{t_{n-1}} \circ \cdots F_{t_1}.
$$
Suppose that the transformations 
satisfy
$$
\sup_{x\in X} d (F(x),  F_t(x)) \le \varepsilon.
$$

Then  there exists $B = B\left(\frac{4 \varepsilon}{1-k},z\right) \subset {X}$ satisfying 
\begin{description}
\item{i)} for any $n>0$ 
$$
F^n_{\bm{t}} (B) \subset B; 
$$
\item{ii)} for all $x\in {X}$ there exists $n=n(x)$ such that
$$
F^n_{\bm{t}} (x) \in B,
$$
where $n$ is uniformly bounded on any compact subset containing $B$.
\end{description}
\end{lemma}

\begin{proof} Consider a fixed $t$ and let $G = F_t$, then note that 
\begin{eqnarray}
d(G(x),G(y)) &=& d(G(x),F(x)) + d(F(y),G(y)) + d(F(x),F(y))    \nonumber \\
&\le&  2 \varepsilon + k d( x , y)  \nonumber
\end{eqnarray}
and,
\begin{eqnarray}
d(G(x), F(y)) &=& d(G(x) ,G(y)) + d( F(y) ,  G(y) )  \nonumber \\
&\le&  3 \varepsilon + k d( x , y  ).  \nonumber
\end{eqnarray}

Now, consider  a ball  $x \in B(z,\delta)$.  We claim that if $ \delta = \frac{3 \varepsilon}{1 - k }$ then 
$F^n_{\bm{t}}(x) \in B(\delta,z)$. Indeed, consider the action of a generic element $G$
\begin{eqnarray}
d( G(x) , z )  &=& d( G(x) ,  F(z) ) \nonumber  \\
&\le& 3 \varepsilon + k d( x , z ) \nonumber \le \delta
\end{eqnarray}
but since $x \in B(\delta,z)$ we have the following bound 
$\delta \ge 3 \varepsilon/(1 - k )$, and by induction 
the claim follows.

To prove the second claim, using the same ideas as before, we observe by induction that 
$$
d(  F^n_{\bm{t}} , z ) \le 3 \varepsilon \left( \sum_{i=0}^{j-1} k^i \right) + k^j  d(  F^{n-j}_{\bm{t}} , z ) .
$$

Hence, 
\begin{eqnarray}
d(  F^n_{\bm{t}} , z ) &\le& 3 \varepsilon \frac{1 - k^n}{1-k} + k^n M \nonumber \\
&=& 3 \varepsilon \frac{1}{1-k} + k^n \tilde M 
\end{eqnarray}
where $\tilde M = M - 3\varepsilon/(1-k)$. This concludes the second part. If $x$ 
is contained in a compact subset $C$ 
of $X$ containing $B$ then the claim on $n = n(x)$ can be made uniform on $x$ 
follows from compactness arguments.  
\end{proof}

Now we are ready to prove the result on perturbations of the transfer operator,  Proposition~\ref{Near}.

\begin{proof}[Proof of Proposition~\ref{Near}] Recall the discussion in Sec. \ref{IsoN}:  the convex cone  $C(a,\nu)$ is endowed with the 
metric 
$$
\theta (\varphi_1, \varphi_2) = \log\frac{\alpha(\varphi_1, \varphi_2)}{\beta(\varphi_1, \varphi_2)}
$$
where
$$
\alpha(\varphi_1 , \varphi_2 ) = \inf \left \{  \frac{\varphi_1(x)}{\varphi_2(x)} , \frac{e^{a d^{\nu}(x,y)} \varphi_2(x) - \varphi_2(y)}
{e^{a d^{\nu}(x,y)} \varphi_1(x) - \varphi_1(y)} \right\} 
$$
and $\beta$ is given by a similar expression with $\sup$ replaced by $\inf$.

Let $\varphi \in C(a,\nu)$ and consider
$\varphi_1 (x) =  (\mathcal{L}_t \varphi)(x)$ and 
$\varphi_2 (x) =  (\mathcal{L} \varphi)(x)$.
For simplicity we write  $\varphi_2 (x) = \varphi_1 (x) + \phi(x)$, where $\phi$ belongs to the cone as well, and 
by construction $\sup | \phi(x) | \lesssim \Delta^{\gamma \nu} $.
Hence, 
$$
\left | \frac{\varphi_1(x)}{\varphi_2(x)} - 1 \right| \lesssim \Delta^{\gamma \nu}  .
$$

Likewise, 
\begin{eqnarray}
\frac{e^{a d^{\nu}(x,y)} \varphi_2(x) - \varphi_2(y)}
{e^{a d^{\nu}(x,y)} \varphi_1(x) - \varphi_1(y)} = 1 + 
\frac{\phi(y)}{\varphi_1(y)}\frac{e^{a d^{\nu}(x,y)} \frac{\phi(x)}{\phi(y)} -1}
{e^{a d^{\nu}(x,y)} \frac{\varphi_1(x)}{\varphi_1(y)}-1}
\end{eqnarray}
but the last fraction in the right hand side is bounded, and therefore, 
\begin{eqnarray}
| \frac{e^{a d^{\nu}(x,y)} \varphi_2(x) - \varphi_2(y)}
{e^{a d^{\nu}(x,y)} \varphi_1(x) - \varphi_1(y)} -1 | \lesssim \Delta^{- \gamma \nu} .
\end{eqnarray}

Therefore, we obtain 
$$
\theta (\varphi_1, \varphi_2) =  \left| \log \left( 1 + \psi(\varphi_1, \varphi_2) \right)  \right| \lesssim \Delta^{-\gamma \nu}
$$
as $\| \psi(\varphi_1, \varphi_2) \|_0 \lesssim \Delta^{-\gamma \nu}$. Proposition \ref{Lfp} guarantees that the transfer operator 
$\mathcal{L}$ is a contraction in the metric space $(C(a,\nu),\theta)$. Moreover, $\mathcal{L}$ and $\mathcal{L}_t$ are uniformly close. Hence, applying  Lemma \ref{ContrF} in the metric space $(C(a,\nu),\theta)$ we conclude Proposition~\ref{Near}.
\end{proof}

{\bf The expectation operator:} As before we consider an absolutely 
continuous measure $\mu$ with density $\varphi$. Moreover, we assume $\log \varphi$ is $(a,\nu)-$H\"older continuous. 
Let $\mu^n : M^n \rightarrow [0,1]$ be a product 
measure, that is,  given $A = A_1 \times \cdots \times A_n \subset M^n$ we have $\mu^n(A) = \mu(A_1) \cdots \mu(A_n)$. 
We define the expectation operator with respect to the measure $\mu$ 
$$
\mathbb{E}_{\mu}(\cdot) = \int_{M^n} \cdot  \mbox{  } ~    d \mu^n \,\, .
$$
The intuitive idea here is that the initial conditions of the low degree nodes are distributed 
according to this product measure. So at the initial time the systems are independent. 
Then with the evolution of the initial conditions, since the systems are interacting, the 
pushforward of this measure under the dynamics no longer has the product 
structure. However, since the interaction is mild  a mean field reduction 
can be obtained.

\begin{propo}[Low degree Nodes] Consider the coupled maps  (\ref{md1}).  
For almost every network in $\mathcal{G}(\bm{w})$, given function $\psi, \sigma \in E$ there 
exists $ u = u(\mu)>0$ and a point  $v \in {\mathbb R}$ such that 
\begin{enumerate}
\item For every $\ell < i \le n$ and $t> u$
$$
\| \mathbb{E}_{\mu}(\psi(x_i(t))) -  v \| \lesssim  \Delta^{-\gamma \nu} ;
$$
\item for any $\ell < j,i \le n$ and $t>u$
$$
\| \mbox{\rm Cov}_{\mu}(\psi(x_i(t)),\sigma(x_j(t)) ) \| \lesssim  \Delta^{-\gamma \nu} .
$$
\end{enumerate}
\label{MeanCoVar}
 \end{propo}

\begin{proof} 
Let $\pi^i : M^n \rightarrow M$ be a projector to the $i$th component, denoting 
$X = (x_1, \cdots, x_i, \cdots, x_n) \in M^n$ then $\pi^i X = x_i$.  
Now the following argument holds for any $\ell < i \le n$, hence, for sake of simplicity 
we drop the $i$ dependence. 
Given $X^0 \in M^n$, take
$$
x^{k+1} = (f_{\bm{t}}^k \circ \pi ) (X^0).
$$

Note that 
$$
\mathbb{E}_{\mu}(\psi(x^k)) = \int_{M^n} \psi(x^k) d \mu^n =  \int_{M^n} \psi \circ f_{\bm{t}}^k \circ \pi d \mu^n 
= \int_{M} \psi \circ f_{\bm{t}}^k d \mu
$$
and notice that $\mu = \pi_* \mu^n$. Since 
$\mu =  \varphi dm$ we may write 
$$
\mathbb{E}_{\mu}(\psi( x^k) ) = \int_{M} \psi \circ f_{\bm{t}}^k d \mu = \int_{M} \psi   ( \mathcal{L}_{\bm{t}}^k \varphi) d m .
$$
By Proposition \ref{Near} we obtain
$$
\mathcal{L}_{\bm{t}}^k \varphi = \varphi_0 + \tilde{\mu}(t)
$$
where $\varphi_0,\tilde{\mu}(t)\in C(a,\nu)$. 
Defining $v = \int_M \psi   \varphi_0 dm \,$  the claim in the first part follows (for every low degree node).

To prove the second part we proceed in the same manner and obtain the desired estimates
for the covariance. Note that 
\begin{eqnarray}
\mbox{Cov}_{\mu}(\psi(x_i^k) ,\sigma( x_j^k) ) &=& \mathbb{E}_{\mu} \left[  \langle \psi( x_i^k)  - \mathbb{E}_{\mu}( \psi( x_i^k) ) , \sigma( x_j^k)  -  \mathbb{E}_{\mu}(\sigma( x_j^k) ) \rangle \right] \nonumber \\
&=& \mathbb{E}_{\mu} \left[ \langle \psi( x_i^k)  , \sigma( x_j^k ) \rangle \right] - \langle \mathbb{E}_{\mu}( \psi(x_i^k)) ,  \mathbb{E}_{\mu}( \sigma( x_j^k) ) \rangle.  \nonumber 
\end{eqnarray}

Now we wish to estimate the $\mathbb{E}_{\mu} \langle \psi( x_i^k)  , \sigma( x_j^k)  \rangle$. To this end 
we fix $i$ and $j$ with $\ell < i \not=j \le n$, and introduce $p_{\bm{t}}^k : M \times M \rightarrow M \times M $
defined as 
$$
p^k (x_i^0, x_j^0) = (f_{\bm{t}}^k(x_i),f_{\tilde{\bm{t}}}^k(x_j)),
$$ 
for the corresponding sequences $\bm{t}$ and $\tilde{\bm{t}}$, 
along with $(\psi ,\sigma ) \circ p^k = (\psi( x_i^k) ,\sigma( x_j^k ))$. 
Consider the projector $\pi^{ij} : M^q \rightarrow M^2$ defined as 
$\pi^{ij} (x_1 , \cdots, x_i, \cdots, x_j, \cdots, x_n) = (x_i,x_j)$. With this notation 
we may write
\begin{eqnarray}
\mathbb{E} \langle \psi( x_i^k) , \sigma( x_j^k ) \rangle &=& \int_{M^n} \langle \psi( x_i^k)  , \sigma( x_j^k)  \rangle d\mu^n \\
&=&\int_{M\times M} \langle \psi(x) , \sigma(y) \rangle \mathcal{L}_{p}^k d\mu^2
\end{eqnarray}
where $\pi^{ij}_* \mu = \mu^2$. Now since $\mu^n$ has a product structure and 
$\pi^{ij}$ is a natural projector $\mu^2$ has density $\varphi(x) \varphi(y)$.  Note that $h$ is in a 
neighborhood of a product map, and all the estimates hold uniformly, we obtain 
$$
(\mathcal{L}_{p}^k \varphi \times \varphi)(x,y) = \varphi_0 (x) \varphi_0 (y) + \bar{\mu}(x,y)  
$$
where again $\bar{\mu}$  satisfies
$$
\| \bar{\mu} \|_0 \lesssim \Delta^{-\gamma  \nu}, 
$$
and the result follows. 
\end{proof}

\subsection{Homogeneity of the Mean Field}

Before, proving the homogeneity of the mean field we need to control 
certain concentration properties of the graphs. The first is the concentration 
of the degrees $k_i = \sum_j A_{ij}$, which is given by Proposition \ref{ConIn}. 
The second is related to
$$
Y_i = \frac{1}{\Delta^2} \sum_{\ell < j,k \le n} A_{ij} A_{ik} .
$$
Next we show that this quantity is heavily concentrated at $\kappa_i^2$. For this we have the following proposition

\begin{propo} Let $\mathcal{G}(\bm{w})$ satisfy the strong heterogeneity hypothesis.  Then 
for $1\le i \le \ell$, for every $0 < \delta < (1 - \theta)/2$ we have 
$$
\left| Y_i - \kappa_i^2 \right| \lesssim \Delta^{(- 1 + \theta)/2 + \delta}
$$
for almost every network in $\mathcal{G}(\bm{w})$.
\end{propo}

\begin{proof} We obtain this claim by a Chebyshev inequality. For $1\le i \le \ell$, for $A_{ij}$ and $A_{ik}$ are 
independent for $j \not =$ and note $A_{ij}^2 = A_{ij}$. Therefore, we need to estimate
the expectation and variance of $Y_i$:
\begin{eqnarray}
\mathbb{E}_{\bm{w}}(Y_i) &=& \frac{w_i}{\Delta^2} \rho \sum_{\ell < j \le n } w_j  +   \frac{w_i^2}{\Delta^2} \rho^2 \sum_{\ell < j\not=k \le n } w_j  w_k. \nonumber
\end{eqnarray}
Note that 
$$
\rho^2 \sum_{\ell < j\not=k \le n } w_j  w_k = \left( \rho \sum_{\ell < j \le n } w_j \right)  \left( \rho \sum_{\ell < k \le n } w_k \right) - \rho^2 \sum_{\ell < j \le n} w_j^2,
$$
along with $\rho^2 \sum_{\ell < j \le n} w_j^2 \le \rho \Delta^{1-\gamma} \le \Delta^{-1-\gamma}$, where in first inequality we used property N2 of the strong heterogeneity and in the last inequality we used the graphical condition  (\ref{DeltaRho}). Moreover, $w_i = \kappa_i \Delta$, and denote $\hat{\kappa} =  (\sum_{j=1}^\ell \kappa_j) / \ell$, clearly $\hat{\kappa} \in (0,1]$. Thus,
$$
{\rho} \sum_{\ell < j \le n} w_j = 1 - {\rho} \ell \hat{\kappa} \Delta .
$$
Using this together with $\Delta^2 \rho \le 1$ and assumption N3 of the strong heterogeneity implies 
$
\left| \rho \sum_{\ell < j \le n} w_i -1 \right| ~ \lesssim ~ \Delta^{-1 + \theta}
$
leading to 
$$
| \rho^2 \sum_{\ell < j\not=k \le n } w_j  w_k - 1 | ~ \lesssim ~  \Delta^{-1 + \theta}, 
$$
and we obtain
\begin{equation}\label{Ye}
\left| \mathbb{E}_{\bm{w}}(Y_i)  - \kappa_i^2 \right| ~ \lesssim ~ \Delta^{-1 + \theta}. 
\end{equation}

Now we wish to estimate, 
\begin{eqnarray}
\mathbb{E}_{\bm{w}}(Y_i^2) &=&  \frac{w_i^4}{\Delta^4} \sum_{\ell < j,k,p,q \le n } \frac{w_j w_k w_p w_q}{ \rho^4} . \nonumber 
\end{eqnarray}
Now by the same arguments as before we obtain
\begin{eqnarray}
| \mathbb{E}_{\bm{w}}(Y_i^2) - \kappa_i^4  | \lesssim \Delta^{-1 + \theta} .
\end{eqnarray}
This bound together with  (\ref{Ye}) yields a bound for the variance 
$$
| \mbox{Var}_{\bm{w}}(Y_i) | \lesssim \Delta^{-1 + \theta} .
$$

Applying the Chebyshev inequality we obtain that 
$$
\mathrm{Pr}_{\bm{w}} \left( | Y_i - \mathbb{E}_{\bm{w}}(Y_i) | ~ \gtrsim ~ \Delta^{-(1+\theta)/2 + \delta}  \right) ~ \lesssim ~ \Delta^{-2 \delta}.
$$
This implies that for almost every graph 
$$
| Y_i - \mathbb{E}_{\bm{w}}(Y_i) | ~ \lesssim ~ \Delta^{-(1+\theta)/2 + \delta},
$$
now using the triangle inequality and the bound  \ref{Ye} we obtain 
$$
| Y_i - \kappa_i^2 | ~ \lesssim ~ \Delta^{-(1+\theta)/2 + \delta},
$$
and we obtain the result. 
\end{proof}

After this auxiliary concentration result we are now ready to state

\begin{propo}[Homogeneity of the mean Field] Let the initial conditions of the low degree nodes be chosen 
independently and according to a measure $\mu^n$ as before. Let $\psi \in E$,  
then there exists ${v}\in {\mathbb R}$ such that given $ \beta>0$ small enough for almost every network in $\mathcal{G}(\bm{w})$
and $\mu^n-$almost every initial condition 
for all $1 \le i \le \ell$ 
$$
\left | \frac{1}{\Delta} \sum_{j=1}^{n} A_{ij} \psi({x}_j(t)) - \kappa_i {v} \right | \lesssim \kappa_i^{1/2} \Delta^{- \eta + \beta}\, \, 
$$
where $\eta>0$ is determined by the network structure and the dynamics:
indeed, 
\begin{itemize}
\item[i)~ ~  ]  $\theta < 1 - \gamma \nu / 2$  \, implies
$$
\eta = \gamma \nu /2;
$$
\item[ii-a)]  $1/2 > \theta > 1 - \gamma \nu / 2$ \, implies 
$$
\eta = 1/2;
$$
\item[ii-b)]  ~$\theta > 1/2 > 1 - \gamma \nu / 2$ \, implies
$$
\eta = 1 - \theta + \frac{1}{2}\frac{\ln \kappa_i}{\ln \Delta}.
$$
\end{itemize}
\label{emf}
\end{propo}

\begin{proof} We wish to use a Chebyshev bound. Hence, we start estimating the mean and the 
variance. For a fixed $i$ in the set of hubs, we need to estimate the expectation and the variance 
of the coupling term $ \Delta^{-1} \sum_{j} A_{ij} \psi( x_j )$.  The first follows easily 
$$
\mathbb{E}_{\mu}\left( \frac{1}{\Delta} \sum_{j} A_{ij} \psi(x_j) \right) = \frac{1}{\Delta} \sum_{j} A_{ij} \mathbb{E}_{\mu}(\psi(x_j)), 
$$
but from Proposition \ref{MeanCoVar}, we have $ | \mathbb{E}_{\mu}\psi(x_j) - v |  \lesssim \Delta^{-\gamma \nu}$. 
Also,  in view of the concentration inequality Proposition \ref{ConIn} for {\it almost every} network
$$
|{k_i} / {\Delta} - \kappa_i | \lesssim \kappa_i^{1/2}\Delta^{- 1/2+ \varepsilon} .
$$ 
Hence, 
\begin{equation}\label{EmuV}
\left | \mathbb{E}_{\mu}\left( \frac{1}{\Delta} \sum_{j} A_{ij} \psi(x_j) \right) - \kappa_i v  \right | \lesssim \max\{ k_i^{1/2} \Delta^{-1/2+\varepsilon}, \kappa_i \Delta^{-\gamma \nu} \}, 
\end{equation}
for almost every network. 
To estimate the variance we note that 
\begin{eqnarray}
\mbox{Var}_{\mu}\left( \frac{1}{\Delta} \sum_{j} A_{ij} \psi(x_j) \right) 
&=& \frac{1}{\Delta^2} \sum_{j,k} A_{ij} A_{ik} [ \mathbb{E}_{\mu} (\psi(x_i) \psi( x_j ) ) - \mathbb{E}_{\mu} (\psi(x_i) )  \mathbb{E}_{\mu} (\psi( x_j)  )  ] \nonumber .
\end{eqnarray}

We split the sum for indexes running over hub nodes $1\le j,k \le \ell$ and low degree nodes $\ell < j,k \le n$.
For the hub indexes we have
\begin{equation}\label{Hb}
\left | \frac{1}{\Delta^2} \sum_{1\le j,k \le \ell} A_{ij} A_{ik} [ \mathbb{E}_{\mu} (\psi( x_i) \psi( x_j) ) - \mathbb{E}_{\mu} (\psi(x_i) )  \mathbb{E}_{\mu} (\psi(x_j ) )  ] \right | \lesssim \Delta^{- 2(1-\theta)} .
\end{equation}

Now for the low degree indexes $\ell < j,k \le n$, we have in view of Proposition \ref{MeanCoVar}
$$
\left | \mathbb{E}_{\mu} (\psi(x_i) \psi( x_j) ) - \mathbb{E}_{\mu} (\psi( x_i))  \mathbb{E}_{\mu} (\psi(x_j) )  \right | \lesssim \Delta^{-\gamma \nu}
$$
and hence, 
\begin{equation}\label{DY}
\left | \frac{1}{\Delta^2} \sum_{1\le j,k \le \ell} A_{ij} A_{ik} [ \mathbb{E}_{\nu} (\psi(x_i) \psi(x_j) ) - \mathbb{E}_{\nu} (\psi(x_i))  \mathbb{E}_{\nu} (\psi(x_j) )  ] \right  | \lesssim \Delta^{-\gamma \nu} Y_i.
\end{equation}
Now, we claim good concentration properties for $Y_i$ so that we can change its value by its expected
value $\kappa_i^2$. Hence, combining (\ref{Hb}) and (\ref{DY}) for {\it almost every} graph we obtain
$$
\left | \mbox{Var}_{\mu}\left( \frac{1}{\Delta} \sum_{j} A_{ij} \psi(x_j) \right)  \right | \lesssim \max \{ \kappa_i^2 \Delta^{  - \gamma \nu } ,  \Delta^{- 2(1-\theta)} \},
$$
therefore, the variance depends on a competition between the network structure parameters. We obtain the following cases:

\medskip
\noindent
{\it Case i) $0<\theta < 1 -\gamma \nu /2$}: Notice that
$$
\left | \mbox{Var}_{\mu}\left( \frac{1}{\Delta} \sum_{j} A_{ij} \psi(x_j) \right)  \right | \lesssim \kappa_i^2 \Delta^{  - \gamma \nu }.
$$

Applying Chebyshev inequality 
$$
\mbox{Pr}_{\mu} \left( \left | \frac{1}{\Delta} \sum_{j=1}^{n} A_{ij} \psi({x}_j(t)) - \mathbb{E}_{\mu}\left(  
 \frac{1}{\Delta} \sum_{j=1}^{n} A_{ij} \psi({x}_j(t)) \right) \right | \gtrsim \kappa_i \Delta^{-\gamma \nu/2 + \beta} \right) \le 
 \Delta^{-2 \beta}.
$$
Therefore, we obtain that for almost every network in $\mathcal{G}(\bm{w})$ and $\mu-$almost every 
initial condition we have
$$
 \left | \frac{1}{\Delta} \sum_{j=1}^{n} A_{ij} \psi({x}_j(t)) - \mathbb{E}_{\mu}\left(  
 \frac{1}{\Delta} \sum_{j=1}^{n} A_{ij} \psi( {x}_j(t)) \right) \right | \lesssim \kappa_i \Delta^{-\gamma \nu/2 + \beta}. 
$$
 But in view of the triangle inequality we obtain 
\begin{equation}\label{Ev}
\left | \left | \mathbb{E}_{\mu}\left(  
 \frac{1}{\Delta} \sum_{j=1}^{n} A_{ij} \psi( {x}_j(t))  \right) - \kappa_i v \right | - \left | 
 \frac{1}{\Delta} \sum_{j=1}^{n} A_{ij} \psi( {x}_j(t) ) - \kappa_i v \right |  \right | \lesssim \kappa_i \Delta^{-\gamma \nu/2 + \beta}
 \end{equation}
however, by (\ref{EmuV}) we obtain 
$$
\left |  \frac{1}{\Delta} \sum_{j=1}^{n} A_{ij}  \psi ( {x}_j(t) )  - \kappa_i v \right | \lesssim \max\{\kappa_i^{1/2} \Delta^{-1/2+\varepsilon}, \kappa_i \Delta^{-\gamma \nu/2 + \beta} \}. 
 $$
Notice that as $\kappa_i \in (0,1]$ we have $ \kappa_i < \kappa_i^{1/2}$. 
Now clearly, $\gamma \nu /2 < 1/2$  as $\nu \in (0,1]$ and $\gamma<1$. Moreover, the competition with the term $\log^{1/2} n$ can be absorbed  in $\beta$.
Combining the two upper bounds estimates we obtain 
$$
 \max\{\kappa_i^{1/2} \Delta^{-1/2+\varepsilon}, \kappa_i \Delta^{-\gamma \nu/2 + \beta} \} \lesssim \kappa_i^{1/2} \Delta^{-\gamma \nu/2 + \beta}
$$
and our first claim follows.

\medskip
\noindent
{\it Case ii) $\theta > 1 - \gamma \nu /2$}. We obtain
$$
\left | \mbox{Var}_{\mu}\left( \frac{1}{\Delta} \sum_{j} A_{ij} \psi(x_j) \right)  \right | \lesssim  \Delta^{- 2(1-\theta)}.
$$
Applying the Chebyshev inequality we obtain an inequality similar to   (\ref{Ev}) with the right hand side replaced by $\Delta^{-(1-\theta) + \beta}$. Hence, we obtain
$$
\left |  \frac{1}{\Delta} \sum_{j=1}^{n} A_{ij}  \psi ( {x}_j(t) )  - \kappa_i v \right | \lesssim \max\{\kappa_i^{1/2} \Delta^{-1/2+\varepsilon}, \Delta^{-(1 - \theta) + \beta} \} \, \, 
 $$
as the condition $\theta > 1 - \gamma \nu / 2$ implies $(1- \theta) <  \gamma \nu$. Here we distinguish 
two cases:
 
 \medskip
\noindent
{\it Case ii-a) $1/2 > \theta > 1 - \gamma \nu /2$}, which implies 
$$
  \max\{\kappa_i^{1/2} \Delta^{-1/2+\varepsilon}, \Delta^{-(1 - \theta) + \beta} \}  \lesssim \kappa_i^{1/2} \Delta^{-1/2+\varepsilon}
$$
and 

\medskip
\noindent
{\it Case ii-b) $ \theta> 1/2 > 1 - \gamma \nu /2$}
$$
  \max\{\kappa_i^{1/2} \Delta^{-1/2+\varepsilon}, \Delta^{-(1 - \theta) + \beta} \}  \lesssim  \Delta^{-(1 - \theta) + \beta}, \, \,
$$
and we conclude the result.
\end{proof}

\subsection{Proof of Theorem \ref{MFR}}

The dynamics of the low degree nodes has been characterized in Proposition \ref{MeanCoVar}. 
To prove our main result we use the Homogeneity of the mean field Proposition \ref{emf}

\begin{proof}
Consider the dynamics of the high degree nodes 
$$
{x}_i(t+1) = { f}({ x}_i(t)) + \alpha r_i(x_i,y_i),\,\,  i=1,\dots,\ell
$$
where the coupling term $r_i$ reads
\begin{equation}\label{rHub}
r_i(x_i,y_i) =  \sum_p u_p(x_i) \left(  \sum_q y_{i,q} \right)
\end{equation}
with 
$$
y_{q,i} = \frac{1}{\Delta}\sum_j A_{ij} v_q(x_j),
$$
as before. By Proposition \ref{emf} on the homogeneity of the mean field,  we obtain 
\begin{equation}
y_{q,i}  =  \kappa_i \langle v_q \rangle +  \xi_{q,i}
\label{eq:ymean}
\end{equation}
for almost all networks in $\mathcal{G}(\bm{w})$, with $ \langle v_q \rangle= \int_M v_q(x)  \mu_0 dm$
and 
\begin{equation}
| \xi_i | \lesssim \kappa_i^{1/2} \Delta^{-\gamma \nu/2 + \beta}, \mbox{ for }i=1,2,\dots,\ell .
\label{eq:errorest}
\end{equation}
Using (\ref{rHub}) and  (\ref{eq:ymean}) we obtain 
$$
r_i(x_i,y_i) = \kappa_i g(x) +  \zeta_i
$$
where 
$$
g(x) =  \sum_{p,q} u_p(x) \langle v_q \rangle= \int h(x,y) \mu_0(d y)
$$ and $\zeta_i = \sum_{p,q} u_p(x_i) \xi_{q,i} $.
Since the functions $u_a$ is continuous and the manifold is compact
and because of  (\ref{eq:errorest}) 
it follows that
$${x}_i(t+1) = { f}({ x}_i(t)) + \alpha \kappa_i g(x_i) + \alpha{\zeta}_i(t), \mbox{ for each }i=1,\dots,
\ell,$$
has the properties claimed in  the Main Theorem.
\end{proof}

{\bf Acknowledgments:} This work was partially supported by FP7 IIF Research Fellowship -- project number 303180,  EU Marie-Curie IRSES Brazilian-European partnership in Dynamical Systems (FP7-PEOPLE-2012-IRSES 318999 BREUDS), and the Brazilian agency FAPESP.


\begin{thebibliography}{99}

\bibitem{Kaneko} K. Kaneko,  Chaos 2, 279 (1992).

\bibitem{Keller} G. Keller, Progress in Probability Volume 46, 183-208 (2000).

\bibitem{Sinai} L.A. Bunimovich and Ya. G. Sinai, Nonlinearity 1, 491 (1988).

\bibitem{BricmontKupiainen} J. Bricmont, A. Kupiainen, Coupled analytic maps, Nonlinearity 8, 379--396 (1995).

\bibitem{Jiang} M. Jiang, Nonlinearity 8, 631--659 (1995). 

\bibitem{Rugh} T. Fischer and H.H. Rugh, Ergodic Theory Dynam. Systems 20, 109--143 (2000).

\bibitem{Jar} E. J\"arvenp\"a\"a, M. J\"arvenp\"a\"a,  Comm. Math. Phys. 220, 1--12, (2001). 

\bibitem{KellerLiverani1} G. Keller and C. Liverani, Discrete Contin. Dyn. Syst. 11,
325--335 (2004).

\bibitem{KellerLiverani2} G. Keller and C. Liverani, 
Comm. Math. Phys. 262, 33--50 (2006).

\bibitem{Book} J.-R.  Chazottes and B. Fernandez (eds.), {\it Dynamics of Coupled Map Lattices and of Related Spatially Extended Systems}, 
Lect. Notes Phys. 671 (Springer, Berlin Heidelberg 2005).

\bibitem{Young} J. Koiller and L-S. Young, Nonlinearity 23, 1121--1141 (2010)

\bibitem{Albert} R. Albert   and A.-L. Barab�si, Rev. Mod. Phys. {\bf 74}, 47 (2002).

\bibitem{Newman} M. E. J. Newman, {\it Networks: An Introduction}, Oxford University Press (2010).


\bibitem{Zhou} C. Zhou and J. Kurths, Chaos {\bf 16}, 015104 (2006); 

\bibitem{Arenas} J. G\'omez-Garde\~nes,
Y. Moreno and A. Arenas, Phys. Rev. Lett. {\bf 98}, 034101 (2007). 

\bibitem{HubSync} T. Pereira, Phys. Rev. E {\bf 82}, 03674 (2010).

\bibitem{Baptista} M. S. Baptista, {\it et al.}, PLoS ONE 7(11): e48118.


\bibitem{HubScience} P. Bonifazi, M. Goldin,  M. A. Picardo, {\it et al.}
Science {\bf 326}, 1419 (2009).

\bibitem{HubEp} R. J. Morgan  and I. Soltesz, Proc. Natl. Acad. Sci. USA {\bf 105}, 6179  (2008). 

\bibitem{Chung} F. R. K. Chung and L. Lu, {\it Complex Graphs and Networks}, American Mathematical Society (2006).

\bibitem{Bollobas} B. Bollob\'as, {\it Random Graphs}, Cambridge University Press; 2 edition (2001). 

\bibitem{ChungComb} F. Chung, L. Lu  and V. Vu, Annals. Comb. {\bf 7}, 21 (2003).

\bibitem{vStrien} W. de Melo and S. van Strien, {\it One-Dimensional Dynamics}, Springer (1993).

\bibitem{Viana} M. Viana, {\it Stochastic dynamics of deterministic systems
}, IMPA (1997).

\end{thebibliography}
\end{document}